\documentclass[11pt]{article}
\usepackage{amsmath, amssymb,epsfig,bm}
\usepackage{latexsym}
\usepackage{graphicx}
\usepackage{color}
\usepackage{algorithm,algorithmic}
\usepackage{CJKutf8}

\numberwithin{equation}{section}

\def\m{\mbox}   
\newcommand{\q}{\quad}

\newtheorem{lemma}{Lemma}[section]

\newtheorem{cor}{Corollary}[section]
\newtheorem{theorem}{Theorem}[section]
\newtheorem{rem}{Remark}[section]
\newtheorem{example}{Example}[section]

\newtheorem{definition}{Definition}[section]

\newcommand{\lla}{\|{\hskip -1pt}|}
\newcommand{\rra}{\|{\hskip -1pt}|}

\newcommand{\vep}{\varepsilon}
\newcommand{\lj}{|{\hskip -1pt} \|}
\newcommand{\rj}{|{\hskip -1pt} \|}

\newcommand{\diam}{\mathrm{diam}}

\newcommand{\sd}{\mathsf{d}}

\newcommand{\be}{\begin{eqnarray}}
\newcommand{\ee}{\end{eqnarray}}
\newcommand{\beq}{\begin{equation}}
\newcommand{\eeq}{\end{equation}}
\newcommand{\ben}{\begin{eqnarray*}}
\newcommand{\een}{\end{eqnarray*}}
\newcommand{\nn}{\nonumber}

\begin{document}

\begin{CJK}{UTF8}{gbsn}

\title{
A Data-Driven Approach to Solving First-Kind Fredholm Integral Equations and Their Convergence Analysis
}

\author{
Duan-Peng Ling\thanks{Department of Mathematics,
Southern University of Science and Technology (SUSTech),
1088 Xueyuan Boulevard, 
University Town of Shenzhen,
Xili, Nanshan, Shenzhen, Guangdong Province, P.R.China. (lingdp2023@mail.sustech.edu.cn).
}
\and Wenlong Zhang\thanks{Corresponding author. Department of Mathematics,
Southern University of Science and Technology (SUSTech),
1088 Xueyuan Boulevard, 
University Town of Shenzhen,
Xili, Nanshan, Shenzhen, Guangdong Province, P.R.China. (zhangwl@sustc.edu.cn). Wenlong Zhang is partially supported by the National Natural Science Foundation of China under grant numbers No.12371423 and No.12241104.
}
}

\date{}
\maketitle

\begin{abstract}
We investigate the statistical recovery of solutions to first-kind Fredholm integral equations with discrete, scattered, and noisy pointwise measurements. Assuming the forward operator’s range belongs to the Sobolev space of order $m$, which implies algebraic singular-value decay $s_j\le Cj^{-m}$, we derive optimal upper bounds for the reconstruction error in the weak topology under an a priori choice of the regularization parameter. For bounded-variance noise, we establish mean-square error rates that explicitly quantify the dependence on sample size $n$, noise level $\sigma$, and smoothness index $m$; under sub-Gaussian noise, we strengthen these to exponential concentration bounds. The analysis yields an explicit a priori and a posteriori rule for the regularization parameter. Numerical experiments validate the theoretical results and demonstrate the efficiency of our practical parameter choice.
\end{abstract}



\section{Introduction}
First--kind Fredholm integral equations
\begin{equation}\label{integral_eq}
  (Kx)(s)\;=\;\int_a^b k(s,t)\,x(t)\,\mathrm{d}t \;=\; y(s),\qquad s\in(a,b),
\end{equation}
constitute a canonical class of linear inverse problems that arise in imaging, geophysics, spectroscopy, and other applied domains \cite{Tikhonov1977, Hansen1998, Engl1996}. When the kernel \(k\) is square--integrable,  the forward operator \(K:L^2(a,b)\to L^2(a,b)\) is compact and typically has a non-closed range; consequently the Moore–Penrose inverse \(K^\dagger\) is discontinuous. The instability of naive inversion has motivated a vast literature on regularization methods, beginning with the classical framework of Tikhonov \cite{Tikhonov1977} and further developed in spectral and variational formulations \cite{Hansen1998, Engl1996, Liu2020}.

Regularization transforms an ill-posed equation into a stable minimization problem.
Among the available schemes, including spectral filtering, iterative regularization, and variational approaches, the Tikhonov functional
\begin{equation}\label{eq:tikhonov}
  J_\alpha(x)=\|Kx-y^\delta\|^2+\alpha\|x\|^2,
\end{equation}
remains fundamental. Its convergence and order-optimality depend on the spectral decay of \(K\) and the source smoothness of the exact solution
\cite{Engl1996, Hofmann2007}. Further extensions introduced discrepancy principles \cite{Morozov1984, Hamrik2001} and heuristic parameter rules \cite{Hanke1996, Hansen1992, lu2013, Liu2020} 
to balance bias and variance automatically.

In modern data-acquisition systems, however, the right-hand side \(y\) is rarely known as a continuous function but rather through discrete and noisy pointwise samples. Such situations naturally lead to statistical inverse problems \cite{Wahba1977, Wahba1990, Tenorio2001, Cavalier2008}, where data are modeled as random variables, and the goal is to estimate the desired solution from stochastic observations. Analyses in this area have characterized minimax-optimal recovery rates under Gaussian and sub-Gaussian noise models
\cite{Bissantz2007, Caponnetto2007, Cavalier2008, Vito2022}. 

Recent studies have extended deterministic regularization to sampled and stochastic frameworks. Slagel et al.~\cite{Slagel2019} develop row-access and sampled-iterative algorithms for computing Tikhonov-regularized solutions in massive or streaming-data regimes, and they propose adaptive parameter-updating rules based on sampled residuals.
Related developments in kernel-based learning theory show that the regularized least-squares estimator achieves spectral-rate optimality in reproducing-kernel Hilbert spaces \cite{Caponnetto2007}, linking the effective dimension of $K$ to the sample complexity of the estimator. To quantify finite-sample fluctuations, modern analyses employ empirical-process and concentration techniques \cite{Vaart1996, Boucheron2013}. These probabilistic tools yield nonasymptotic exponential deviation bounds for regularized estimators under sub-Gaussian noise, strengthening classical mean-square results. Such inequalities have been established for kernel-based Tikhonov estimators in random-design models
\cite{Smale2007} and for empirical risk minimization in inverse problems \cite{klemela2010}, offering precise high-probability risk control and guiding data-driven parameter selection. A parallel strand of research has developed stochastic convergence theory for inverse problems driven by observational data and their numerical discretizations (see, e.g., \cite{Chen2018, Chen2020, Chen2022, Wang2023, Sun2024, Cen2025, Sun2025a, Gu2025, Jin2025, Zhang2025}). Despite these advances, however, systematic nonasymptotic stochastic convergence results specifically for first-kind Fredholm integral equations under discrete, scattered sampling remain limited. 

Building on this foundation, the present paper investigates the stochastic convergence behavior of the first-kind Fredholm integral equation under scattered but quasi-uniform noisy observations. Specifically, we consider an integral operator \(K: L^2(a, b) \to L^2(a, b)\) defined by (\ref{integral_eq}), where the kernel \(k(s,t) \in L^2((a, b)^2)\) and the range of \(K\) lies within the Sobolev space \(H^m(a,b)\) for some positive integer \(m\). Under this assumption, \(K\) is a compact operator with a non-closed range \(R(K)\), implying that the integral equation (\ref{integral_eq}) is inherently ill-posed. In our setting, measurements of the right-hand side are available only at discrete spatial points and are corrupted by random noise, without imposing additional smoothness constraints on the data. The observed data are given by \(\boldsymbol w(s) = y(s) + e(s), ~ s = s_1, s_2, \cdots, s_n,\) where the \(e = (e(s_1), e(s_1), \cdots, e(s_n))^T\) is the data noise vector, with $\{e(s_i)\}_{i=1}^n$ being independent and identically distributed random variables on a probability space $(\mathcal{X}, \mathcal{F}, \mathbb{P})$. In this paper, we examine two types of random noise sequences \(\{e(s_i)\}_{i=1}^n\). In the first case, the random variables \(\{e(s_i)\}_{i=1}^n\) are independent with zero mean, \(\mathbb{E}[e(s_i)] = 0\), and uniformly bounded variance, \(\operatorname{Var}[e(s_i)] \le \sigma^2\).  In the second case, the noise variables are assumed to be independent sub-Gaussian random variables with parameter \(\sigma\). 

To stabilize the ill-posed problem, we use a variant of Tikhonov regularization. For given measurements data vector \(\boldsymbol w = y + e\), the regularized solution \( x_{n,\alpha} \) to equation (\ref{integral_eq}) is defined by solving the following minimization problem: find \( x \) in \( L^2(a,b) \) which minimizes
\begin{equation}\label{Tikhonov reg}
    |Kx - \boldsymbol w |_n^2+\alpha \|x\| ^2.
\end{equation}
Here, \(\alpha > 0 \) is called a regularization parameter and \(|\cdot|_n\) denote the empirical semi-norm \(|u|_n = \left(\frac{1}{n} \sum_{i=1}^{n} u^2(x_i)\right)^{1/2}\) for any \( u \in C([a, b])\), which induced by the semi-inner product \((u,v)_n = \frac{1}{n} \sum_{i=1}^{n} u(x_i) v(x_i)\) for any \(u,v \in C([a, b])\). Throughout this paper, \((\cdot,\cdot)\) and \(\|\cdot\|\) denote, respectively, the standard inner product \((\cdot,\cdot)_{L^2(a,b)}\) and norm \(\|\cdot\|_{L^2(a,b)}\) in the Hilbert space \(L^2(a,b)\), unless otherwise specified. We use the notation \(C\) to represent a generic positive constant that is independent of the mesh size \(h\) and the regularization parameter \(\alpha\); its value may vary from one occurrence to another.

Against this background, the main contributions of the present paper are summarized as follows:
\begin{itemize}
    \item We establish a unified statistical–analytic framework that provides sharp nonasymptotic mean-square convergence rates for sampled Tikhonov reconstructions with an explicit a priori rule for the regularization parameter \(\alpha\), capturing its dependence on the sample size \(n\), noise level \(\sigma\), and smoothness index \(m\). Convergence is analyzed within an appropriately chosen Hilbert framework \((W, W^*)\) linked to \((K^*K)^{1/4}\), providing stochastic rate estimates that connect the attainable convergence of first-kind problems with discrete data directly to the spectral decay of the forward operator. The analysis extends to high-probability exponential concentration under sub-Gaussian noise using empirical-process and entropy bounds. This framework delivers precise finite-sample guarantees, both in expectation and with explicit tail bounds, while elucidating the attainable convergence behavior and guiding practical parameter selection for first-kind problems with discrete observations.

    \item We remove the common reliance on classical source conditions,  i.e., that the best-approximation solution \(x^\dagger\) lies in a range of a power of \(K^*K\), and derive finite-sample convergence rates and high-probability bounds under substantially weaker, directly verifiable assumptions on the forward operator’s range regularity and on the noise or sampling model. This removes a standard theoretical restriction, broadens applicability to solutions lacking classical source regularity, and still yields operator-spectral rate characterizations and practical guidance on parameter choice.
    \item We propose a practical, data-driven parameter-selection procedure that is implementable from discrete observations, prove its monotonicity properties, and demonstrate through extensive numerical experiments that the rule attains near-optimal performance consistent with the theoretical rates.
\end{itemize}

The remainder of the paper is organized as follows. Section 2 establishes the variational characterization of the regularized solution to the first-kind Fredholm integral equation. Section 3 develops stochastic convergence in expectation under bounded-variance noise. Section 4 treats sub-Gaussian noise and obtains exponential tail bounds via empirical-process techniques. Section 5 reports numerical experiments that validate the theory and illustrate implementation details. Throughout, the analysis highlights how forward regularity $m$, sample size $n$, and noise structure jointly determine attainable stochastic convergence rates and practical parameter choices.

\section{Variational characterization of the data driven regularized solution}

To proceed, we establish the variational characterization of the regularized solution $x_{n,\alpha}$ to equation (\ref{integral_eq}).

\begin{lemma}\label{min_var_th}
Let $K:L^2(a,b)\to L^2(a,b)$ be the integral operator $(Kx)(s)=\int_a^b k(s,t)x(t)\,dt$ with the kernel $k(s,t)\in L^2((a,b)^2)$ and let observations be $w_i=y(s_i)+e(s_i)$ at sampling points $\{s_i\}_{i=1}^n$.  
Then for $\alpha>0$, $x_{n,\alpha}$ is the unique minimizer of the functional
    \begin{equation}\label{min_fun}
        \min_{x\in L^2(a,b)} J_{n,\alpha} (x) = |Kx - \boldsymbol w|_n^2 + \alpha \|x\|^2.
    \end{equation}
    \noindent It is characterized in variational form 
\begin{equation}\label{variational form}
     (K x_{n,\alpha}, K v)_n + \alpha (x_{n,\alpha}, v) = (\boldsymbol w, K v)_n, \quad \forall v \in L^2(a, b).
\end{equation}
\end{lemma}

\begin{proof}
      For $\alpha > 0$, $J_{\alpha}$ is strictly convex, since semi-norm \(|\cdot|_n\) is convex and \(L^2\)-norm \(\|\cdot\| \) is strictly convex. If \(\|x\|_{L^2}\) grows without bound, then \(J_{n,\alpha}(x) \to \infty\). This shows that \(J_{n,\alpha}\) is coercive.
       Moreover, the term \(\|x\| ^2\) is weakly lower semicontinuous in a Hilbert space and the empirical semi-norm squared term \( |Kx - \boldsymbol w|_n^2 \) is also weakly lower semi-continuous since \(K\) is bounded linear, so that \(J_\alpha\) are weakly lower semicontinuous. Hence, \(J_{n,\alpha}\) has a unique minimizer \(x_{n, \alpha}\).

    Consider    
    \[\varphi(t) = |K(x_{n,\alpha} + tv) - \boldsymbol w|_n^2 + \alpha \|x_{n,\alpha} + tv\| ^2,\quad \text{for all}~~ v \in L^2(a, b), t\in \mathbf{R}.\]
    Then \(\varphi'(0) = 0\), as \(\varphi(t)\) achieves its minimum at \(t = 0\). It follows that variational form (\ref{variational form}) holds. 
\end{proof}

\begin{lemma}\label{min-problem}
    Given any \( \boldsymbol w = (w_1, w_2, \cdots, w_n)^T \in \mathbb{R}^n \), let \( u \) be the solution that minimizes
        \[
        \|u\| ^2 \quad \text{subject to} \quad (Ku)(s_i)=w_i.
        \]
    Then \( u \in V_n \), where $V_n$ is a n-dimensional subspace of $N(K)^\bot$.
\end{lemma}

\begin{proof} 
    For any \(u \in L^2(a,b)\) satisfied \((Ku)(s_i) = w_i\), we have 
        \[u = u_1 + u_2 \quad \text{and} \quad \|u\| ^2 = \|u_1\| ^2 + \|u_2\| ^2,\quad u_1 \in N(K),~ u_2  \in N(K)^\bot\]
    and hence
        \[(Ku)(s_i) = (Ku_2)(s_i) = w_i,\]
    where we have used the fact that
        \(L^2(a,b) = N(K) \oplus N(K)^\bot\).
    Thus, the optimization problem
        \[\min_{u\in L^2(a,b),(Ku)(s_i)=w_i} \|u\|^2  = \min_{u\in N(K)^\bot,(Ku)(s_i)=w_i} \|u\|^2.\]
    
    Define
        \[\widetilde{K} := K|_{N(K)^\bot}: N(K)^\bot \rightarrow H^m(a,b).\]
    Then \(\widetilde{K}\) is injective. Let $V$ be a subset of $N(K)^\bot$ such that
    \begin{equation}
        V=\{v\in N(K)^\bot: (\widetilde{K}v)(s_i)=0, i=1,2,\cdots, n\}\,.
    \end{equation}
    One can see that \(V\) is a closed linear subspace.
    
    Define the operator $P_V: N(K)^\bot \rightarrow V$,
    \begin{equation*}
        (P_V [u],v)=(u,v),  \quad \forall v\in V.
    \end{equation*}
    It is easy to check that \(P_V\) is linear, idempotent and self-adjoint and hence is an orthogonal projection of $N(K)^\bot$ onto $V$. Then \(I - P_V\) is an orthogonal projection of $N(K)^\bot$ onto the closed linear subspace $V^\bot$, where the notation \( I \) is identity map.
    Take $\phi_i\in N(K)^\bot$ such that $(\widetilde{K}\phi_i)(s_j)=\delta_{i,j}$, where $\delta_{i,j}$ is the Kronecker delta function.
    Denoting $\psi_i= (I- P_V)\phi_i$, we have $(\widetilde{K}\psi_i)(s_j)=(\widetilde{K}\phi_i)(s_j)- (\widetilde{K}(P_V\phi_i))(s_j) = (\widetilde{K}\phi_i)(s_j) = \delta_{i,j}$, which implies that $\psi_1, \psi_2,\cdots,\psi_n$ are linearly independent. Thus $V_n=\m{span}\{\psi_1, \cdots, \psi_n\}$ is a n-dimensional subspace of $N(K)^\bot$. 
    
    For any $u\in N(K)^\bot$, define the interpolation operator $J$
    \begin{equation*}
        Ju=\sum_{i=1}^n ((\widetilde{K}u)(s_i)) \psi_i, 
    \end{equation*}
    which implies \((\widetilde{K}(Ju))(s_j) = (\widetilde{K}u)(s_j)\). It can easily be verified from the definition above that $Ju\in  V_n \subset V^\bot$ and $u-Ju\in V$, meaning \((Ju,u-Ju) = 0\). It follows that \(\|u\|^2 - \|Ju\|^2 = \|u-Ju\|^2 \ge 0\).
    Therefore,
    \begin{equation*}
        \min_{u\in N(K)^\bot,(Ku)(s_i)=w_i} \|u\|^2 = \min_{(\widetilde{K}u)(s_i)=w_i} \|u\|^2 =\min_{u\in V_n,(Ku)(s_i)=w_i} \|u\|^2.
    \end{equation*}
     
\end{proof}
\section{Mean-square and weak convergence analysis with bounded-variance noise}\label{sec3}

In this section, we study the integral operator \(K: L^2(a,b) \to L^2(a,b)\) defined by (\ref{integral_eq}), where the kernel \(k(s,t) \in L^2((a,b)^2)\) and the range of \(K\) is contained in the Sobolev space \(H^m(a,b)\) for some positive integer \(m\). The observed data are contaminated by independent random errors \(\{e(s_i)\}_{i=1}^n\) satisfying \(\mathbb{E}[e(s_i)] = 0\) and \(\mathrm{Var}[e(s_i)] \le \sigma^2\). We analyze the stochastic convergence of the regularized solution \(x_{n,\alpha}\) in terms of the expected errors \(\mathbb{E}[|Kx_{n,\alpha} - Kx^\dagger|_n]\) and \(\mathbb{E}[\|x_{n,\alpha} - x^\dagger\|_{W^*}]\), showing that \(x_{n,\alpha}\) converges to \(x^\dagger\) in expectation under the weak topology.

\subsection{Preliminaries}

We review the basic definitions and fundamental properties of approximation numbers, which will serve as indispensable tools in our subsequent analysis of the eigenvalues or singular values of bounded linear operators. Throughout, we write $L(A,B)$ for the space of all bounded linear operators from a Banach space $A$ to a Banach space $B$.

\begin{definition}[\cite{Edmunds2008}, p.11]
Let $A$ and $B$ be Banach spaces and let $T\in L(A,B)$.  For each $j\in\mathbb{N}$, the $j$th approximation number of $T$ is defined by
$$
a_j(T)\;=\;\inf\bigl\{\|T - L\|\;:\;L\in L(A,B),\;\mathrm{rank}\,L<j\bigr\},
$$
where $\mathrm{rank}\,L$ is the (finite) dimension of the range of $L$.  
\end{definition}

Intuitively, $a_j(T)$ measures the smallest operator‐norm error incurred when approximating $T$ by any operator of rank less than $j$.
The following lemma, adapted from the book \cite{Edmunds2008}, collects the fundamental properties of these numbers.
\begin{lemma}[\cite{Edmunds2008}, Lemma 2, p.11]\label{pro_appro_num}
Let $A$, $B$, and $C$ be Banach spaces, let $S,T\in L(A,B)$, and let $R\in L(B,C)$.  Then for every $i,j\in\mathbb{N}$:
\begin{enumerate}
    \item[(i)] $\|T\| = a_1(T) \geq a_2(T) \geq \cdots \geq 0.$
    \item[(ii)] $ a_{i+ j-1}(R \circ S) \leq a_i(R)a_j(S).$
    \item[(iii)] If $B$ is a $p$-Banach space $(0 < p \leq 1)$, then
    \[
    a_{i+j-1}^p(S + T) \leq a_i^p(S) + a_j^p(T).
    \]
\end{enumerate}
\end{lemma}
 
In the special case of Hilbert spaces, approximation numbers enjoy particularly transparent interpretations, as they coincide with the singular values of compact operators and thus admit direct spectral characterizations.

\begin{theorem}[\cite{Edmunds1987}, Theorem 5.10, p.88]\label{appro_s-num}
Let \(H\) and \(H_1\) be infinite-dimensional Hilbert spaces and $T: H \to  H_1$ is a compact operator. Then, for all $j \in \mathbb{N}$, we have 
$$
a_j(T)\;=\;s_j(T),
$$
where $s_j(T)$ denotes the $j$th singular value of $T$.
\end{theorem}

We next state a sharp estimate for the approximation numbers of the Sobolev embedding on a smooth bounded domain.

\begin{theorem}[\cite{Edmunds2008}, p.119]\label{app_bound}
Let $\Omega$ be a bounded domain in $\mathbb{R}^n$ with $C^\infty$ boundary and 
\[
\text{id} : H^m(\Omega) \to L^2(\Omega)
\]
be the natural embedding. 
Then there are positive numbers $C_1$ and $C_2$ such that for all $j \in \mathbb{N}$,
\[
C_1 j^{-m} \leq a_j(id) \leq C_2 j^{-m}. 
\]
\end{theorem}

Finally, we turn to some basic facts about the Fredholm integral equation, which will serve as the basis for the subsequent analysis.

\begin{theorem}[\cite{Kirsch2021}, p.324]\label{thm:int_op}
Let \( k(s,t) \in L^2((a,b)^2) \). The operator
\[
(Kx)(s) := \int_a^b k(s,t) x(t) \, dt, \quad s \in (a,b), \quad x \in L^2(a,b),
\]
is well-defined, linear, and bounded from \(L^2(a,b)\) into \(L^2(a,b)\). Furthermore,
\[
\|K\|_{\mathcal{L}(L^2(a,b),L^2(a,b))} \leq 
\sqrt{\int_c^d \int_a^b |k(s,t)|^2 \, ds \, dt}.
\]
\end{theorem}

\begin{theorem}[\cite{Kirsch2021}, p.331]\label{K_compact}
    
Let \( k \in L^2((a,b)^2) \). The operator \( K : L^2(a,b) \to L^2(a,b) \), defined by
\[
(Kx)(t) := \int_a^b k(s,t) x(t) \, dt, \quad s \in (a,b), \quad x \in L^2(a,b),
\]
is compact from \(L^2(a,b)\) into \(L^2(a,b)\).
\end{theorem}


\subsection{Convergence analysis in probability}
Let $\Delta = \Delta_n = \{a \leq s_1 < \cdots < s_n \leq b\}$ denote an $n$-point partition of $[a, b]$.  Assume the partition is quasi-uniform, i.e., there exists a constant \( B \) such that \(h_{max}/h_{min} \le B \), where \(h_{max}\) and \(h_{min}\) are defined by
\begin{equation}
\begin{aligned}
h_{\max} = \max_{x \in [a, b]} \min_{1 \leq i \leq n} |s - s_i| \quad \text{and} \quad h_{\min} = \min_{1 \leq i \neq j \leq n} |s_i - s_j|.
\end{aligned}
\end{equation}
By the paper \cite[Theorem 3.3 and 3.4]{Utreras1988}, there exists a constant \(C >0\) depending only \(a, b, m, B \) and \(h_0\) such that for any for any $y \in H^m(a, b)$ and $h :=h_{\max} \le h_0$, we have
\begin{equation}\label{relationship between the norms}
\begin{aligned}
    \|y\|^2 \leq C(|y|_n^2 + h^{2m} \|y\|_{H^m(a, b)}^2), \quad |y|_n^2 \leq C(\|y\|^2 + h^{2m} \|u\|_{H^m(a, b)}^{2}).
\end{aligned}
\end{equation}

We start with the decay rate of the singular values of the integral operator $K$.
\begin{theorem}\label{s-val decay}
    Consider the integral operator $K: L^2(a,b) \to L^2(a,b)$ defined by
$$
\bigl(Kx\bigr)(t)\;=\;\int_a^b k(s,t)\,x(t)\,\mathrm{d}t,\quad s\in(a,b),
$$
with the kernel $k(s,t)\in L^2\bigl((a,b)^2\bigr)$, and suppose further that the range of $K$ is contained in the Sobolev space $H^m(a,b)$ for some positive integer $m$. 
Then the singular values of \(K\) satisfy the estimate
    \begin{equation}
        s_j(K) \le Cj^{- m},
    \end{equation}
where \(C\) is a constant independent of \(j\) and \(s_j(K)\) are placed in non-increasing order.
\end{theorem}
\begin{proof}
By $k(s,t)\in L^2\bigl((a,b)^2\bigr)$, it follows from Theorem \ref{thm:int_op} that the integral operator $K: L^2(a,b) \to L^2(a,b)$ is bounded linear. Given that \(R(K) \subset H^m(a,b)\), we define the bounded linear operator
   $$
   T: L^2(a,b)\;\rightarrow\;H^m(a,b),\quad T(x)=K(x).
   $$
Let 
   $$
   id: H^m(a,b)\;\hookrightarrow\;L^2(a,b),
   $$
   be the natural Sobolev embedding. Then $K$ can be represented as the composition $K = id\circ T$. 
By Lemma  \ref{pro_appro_num}(i) and (ii), the approximation numbers of \(K\) satisfy
   $$
   a_j(K)=a_j\bigl(id\circ T\bigr)\;\le\;\|T\|\, a_j(id).
   $$
From Theorem \ref{K_compact}, the operator \(K\) is compact and hence by Theorem \ref{appro_s-num}, we have $a_j(K)=s_j(K)$. 

Applying Theorem \ref{app_bound}, there exists a positive \(C\) such that 
    $$s_j(K) \le C \, j^{-m}. $$
   
\end{proof}
\subsubsection{Convergence analysis with respect to expectation}
Now, we consider the eigenvalue problem of the operator \(K\) in the n-dimensional subspace \(V_n\).
\begin{lemma}\label{eigenvalue-problem} 
    Assume that \(\widetilde{K}\) and \(V_n\) are defined as in Lemma \ref{min-problem}. Consider the eigenvalue problem
    \beq\label{eigen-n}
    (\psi, v) = \lambda\, (\widetilde{K}\psi, \widetilde{K}v)_n \quad \text{for all } v \in V_n.
    \eeq
    This problem admits exactly \(n\) finite eigenvalues, denoted by \(\lambda_1 \le \lambda_2 \le \cdots \le \lambda_n\), and the associated eigenfunctions form an orthogonal basis for \(V_n\) under the norm \(\|\widetilde{K}\cdot\|_n\). Moreover, there exists a constant \(C > 0\), independent of \(j\), such that
    \[
    \lambda_j \ge C\, j^{2m} \quad \text{for } j = 1, 2, \dots, n.
    \]
\end{lemma} 

\begin{proof} 
    From the proof of the Lemma \ref{min-problem}, we know $(\widetilde{K}\psi_i)(s_j)=\delta_{i,j}$. Thus, any \(\psi\in V_n\) can be expressed as
        \[
        \psi=\sum_{i=1}^n \bigl((\widetilde{K}\psi)(s_i)\bigr)\psi_i,
        \]
    which shows that \(\|\widetilde{K}\cdot\|_n\) defines a norm on \(V_n\).
    Define the vector 
        \[
        \vec{\psi} = \Bigl((\widetilde{K}\psi)(s_1), \dots, (\widetilde{K}\psi)(s_n)\Bigr)^T,
        \]
    and let the matrices \(A=(a_{i,j})\) and \(B=(b_{i,j})\) be given by
        \[
        a_{i,j} = (\psi_i,\psi_j),\quad \text{and} \quad b_{i,j} = (\widetilde{K}\psi_i,\widetilde{K}\psi_j)_n = \frac{1}{n}\sum_{k=1}^{n}\delta_{i,k}\delta_{j,k}.
        \]
    Then the eigenvalue problem (\ref{eigen-n}) can be rewritten as
    \begin{equation}\label{mat-eigen-pro}
        A \vec{\psi} = \lambda\,B \vec{\psi}.
    \end{equation}
    Since \(A\) is symmetric and positive definite and \(B = \frac{1}{n}I\), this reformulated problem (\ref{mat-eigen-pro}) has \(n\) positive eigenvalues and \(n\) corresponding eigenvectors that form an orthogonal basis of \(\mathbb{R}^n\) with respect to the \(l^2\)-norm. Noting that
    \[
    \|\vec{\psi}\|_{l^2}=\sqrt{n}\,\|\widetilde{K}\psi\|_n,
    \]
    we deduce that the original eigenvalue problem \eqref{eigen-n} possesses \(n\) finite eigenvalues \(\lambda_1,\dots,\lambda_n\) and the associated eigenfunctions form an orthogonal basis of \(V_n\) under the norm \(\|\widetilde{K}\cdot\|_n\).

    Next, we derive a lower bound for the eigenvalues \(\lambda_j\). Since \(\widetilde{K}\) is a bounded linear operator from \(N(K)^\perp\) to \(H^m(a,b)\), it follows from \eqref{relationship between the norms} that
    \begin{equation}\label{n and L^2 norm}
        |\widetilde{K}u|_n^2 \leq C\Bigl(\|\widetilde{K}u\|^2 + h^{2m}\|u\|^2\Bigr),\quad \forall\, u\in N(K)^\perp.
    \end{equation}
    By applying the min-max principle for the Rayleigh quotient, together with \eqref{n and L^2 norm} and Theorem \ref{s-val decay}, we obtain
    \begin{align*}
        \lambda_j
        &=\min_{dim(X)=j,X\subset V_n}\max_{u\in X} \frac{(u, u)}{(\widetilde{K}u,\widetilde{K}u)_n}\\
        &\geq C \min_{dim(X)=j, X\subset V_n}\max_{u\in X} \frac{(u, u)}{(\widetilde{K}u,\widetilde{K}u) + h^{2m}(u,u)}\\
        &\geq  C \min_{dim(X)=j, X\subset N(K)^\bot}\max_{u\in X} \frac{(u, u)}{(\widetilde{K}u,\widetilde{K}u)_n + h^{2m}(u,u)}\\
        &=C \frac{1}{s_j^{2}+h^{2m}}\\
        &\geq C \frac{1}{j^{-2m}+h^{2m}}.  
    \end{align*}
    Since the mesh \(\Delta\) is quasi-uniform, there exists a constant \(C\) such that \(h\leq C\frac{1}{n}\). Moreover, since \(j\leq n\), we have \(j^{-2m}\geq n^{-2m}\). This yields the desired lower bound for \(\lambda_j\).
     
\end{proof}

For any \(v \in L^2(a,b)\), we define the energy norm by
\[
\lla v\rra_{\alpha}^2 := \alpha\,\|v\|^2 + |Kv|_n^2.
\]
Using this energy norm, we derive an upper bound for both the regularized solution and the best-approximate solution of the equation (\ref{integral_eq}).

\begin{lemma}\label{bound for reg-sol and min-sol}
    Let $x_{n,\alpha}$ and $x^\dagger$ be the regularized solution and best-approximate solution of equation (\ref{integral_eq}) respectively. Then
    \begin{equation}\label{estimate under energy norm}
        \lj x_{n,\alpha} - x^\dagger\rj_\alpha \le \alpha^\frac{1}{2}\|x^\dagger\| + \sup_{v \in L^2(a, b)}\frac{(e, K v)_n}{\lj v\rj_\alpha}.
    \end{equation}
\end{lemma}
\begin{proof}
    By assumption, \(\boldsymbol w(s_i) = (Kx^\dagger)(s_i) + e(s_i)\). Inserting this into (\ref{variational form}), we obtain 
    \begin{equation}
        (K (x_{n,\alpha} - x^\dagger), K v)_n + \alpha (x_{n,\alpha}, v) = (e, K v)_n \quad \forall v \in L^2(a, b).
    \end{equation}
    Taking \(v = x_{n,\alpha} - x^\dagger\), meaning \(x_{n,\alpha} = v + x^\dagger\), we get 
    \begin{equation}
        |K v|_n^2 + \alpha(v,v) + \alpha (x^\dagger, v) = (e, K v)_n.
    \end{equation}
    Then
    \begin{align*}
        \lj v\rj_\alpha^2 &\le (e, K v)_n + \alpha |(x^\dagger, v)| \\
                        &\le (e, K v)_n + \alpha^\frac{1}{2}\|x^\dagger\| \cdot\alpha^\frac{1}{2}\|v\| \\
                        &\le (e, K v)_n + \alpha^\frac{1}{2}\|x^\dagger\| \cdot \lj v\rj_\alpha.
    \end{align*}
    It follows that \(\lj v\rj_\alpha \le \alpha^\frac{1}{2}\|x^\dagger\| +\frac{(e, K v)_n}{\lj v\rj_\alpha}\). 
    Thus,
    \(\lj x_{n,\alpha} - x^\dagger\rj_\alpha \le \alpha^\frac{1}{2}\|x^\dagger\| + \sup_{v \in L^2(a, b)}\frac{(e, K v)_n}{\lj v\rj_\alpha}.\)
    
\end{proof}

\begin{theorem}\label{boundedness_x_n}
Let $x_{n,\alpha}$ and $x^\dagger$ be defined as lemma \ref{bound for reg-sol and min-sol}. Then there exist constants $\alpha_0 > 0$ and $C>0$ such that for any $\alpha_n \leq \alpha_0$,
\begin{align}
    \mathbb{E} \big[\|x_{n,\alpha} - x^\dagger\|^2\big] &\leq C \|x^\dagger\|^2 + \frac{C\sigma^2}{n\alpha^{1+\frac{1}{2m}}} \label{ineq1}\\
    \mathbb{E} \big[|Kx_{n,\alpha} - Kx^\dagger|^2_n\big] & \leq C \alpha \|x^\dagger\|^2 + \frac{C\sigma^2}{n\alpha^{\frac{1}{2m}}} \label{ineq2}.
\end{align}
\end{theorem}

\begin{proof} 
By the definition of energy norm, we have
\begin{equation}\label{energy-estimate}
    \|x_{n,\alpha} - x^\dagger\|^2 \le \alpha^{-1} \lla x_{n,\alpha} - x^\dagger\lla_\alpha^2 \quad \text{and} \quad |Kx_{n,\alpha} - Kx^\dagger|_n^2 \le \lla x_{n,\alpha} - x^\dagger\rra_\alpha^2.
\end{equation}
For the supremum term of inequality (\ref{estimate under energy norm}), we have 
\begin{equation}\label{identity}
    \sup_{v \in L^2(a, b)}\frac{(e, K v)_n^2}{\lj v\rj_\alpha^2} \le \sup_{v\in V_n}\frac{(e,\widetilde{K}v)^2_n}{\lj v \rj_\alpha^2},
\end{equation}
where $\widetilde{K}$ is defined as Lemma \ref{min-problem}. In fact, 
\begin{align*}
    \sup_{v \in L^2(a, b)}\frac{(e, K v)_n^2}{\lj v\rj_\alpha^2} & = \sup_{v\in N(K)^\bot}\frac{(e,\widetilde{K}v)_n^2}{\alpha(v,v)+|\widetilde{K}v|_n^2}\\
    & \le \sup_{v\in N(K)^\bot}\frac{(e,\widetilde{K}v)^2_n}{\alpha \min_{u\in N(K)^\bot, (\widetilde{K}u)(s_i)=(\widetilde{K}v)(s_i)}(u,u)+|\widetilde{K}v|_n^2}\\
    &= \sup_{v\in N(K)^\bot}\frac{(e,\widetilde{K}v)^2_n}{\alpha \min_{u\in V_n, (\widetilde{K}u)(s_i)=(\widetilde{K}v)(s_i)}(u,u)+|\widetilde{K}v|_n^2}\\
    &= \sup_{v\in V_n}\frac{(e,\widetilde{K}v)^2_n}{\alpha \|v\|^2 + |\widetilde{K}v|_n^2} = \sup_{v\in V_n}\frac{(e,\widetilde{K}v)^2_n}{\lj v \rj_\alpha^2}.
\end{align*}
where we have used Lemma \ref{min-problem}. 

From Lemma \ref{eigenvalue-problem}, we know that eigenvalues of the problem (\ref{eigen-n}) have exactly \(n\) finite eigenvalues \(\lambda_1 \le \lambda_2 \le \cdots \le \lambda_n\), and the associated eigenfunctions $\{\psi_j\}^n_{j=1}$ form an orthonormal basis for \(V_n\) under the norm \(\|\widetilde{K}\cdot\|_n\), i.e., $\|\widetilde{K}\psi_j\|_n=1$ and $(\widetilde{K}\psi_i,\widetilde{K}\psi_j)_n=\delta_{ij}$ for $i,j=1,2,\cdots, n$. Then $v(x)=\sum^n_{j=1}v_j\psi_j(x)$ for any $v\in V_n$, where $v_j=(\widetilde{K}v,\widetilde{K}\psi_j)_n$ for $j=1,2,\cdots,n$. Thus, we can derive \(\lla v\rra^2_{\alpha}=\sum^n_{j=1}(\alpha\lambda_j+1)v_j^2.\)

For the numerator of the right-hand side of the inequality (\ref{identity}), we have 
\begin{align*}
    (e,\widetilde{K}v)_n^2 &= \frac{1}{n^2} \left( \sum_{i=1}^n e(s_i) \left( \sum_{j=1}^n v_j (\widetilde{K} \psi_j)(s_i) \right) \right)^2 \\
    &= \frac{1}{n^2} \left( \sum_{j=1}^n (1 + \alpha \lambda_j)^{\frac{1}{2}}v_j\cdot (1 + \alpha \lambda_j)^{-\frac{1}{2}} \left( \sum_{i=1}^n e(s_i) (\widetilde{K} \psi_j)(s_i) \right) \right)^2 \\
    &\leq \frac{1}{n^2} \sum_{j=1}^n (1 + \alpha \lambda_j) v_j^2 \cdot \sum_{j=1}^{n}(1 + \alpha \lambda_j)^{-1} \left( \sum_{i=1}^n e(s_i) (\widetilde{K} \psi_j)(s_i) \right)^2 \\
    &= \lj v \rj_\alpha^2 \cdot \frac{1}{n^2}\sum_{j=1}^{n}(1 + \alpha \lambda_j)^{-1} \left( \sum_{i=1}^n e(s_i) (\widetilde{K} \psi_j)(s_i) \right)^2.
\end{align*}
Thus,
\begin{align*}
    \mathbb{E}\left[\sup_{v\in V_n}\frac{(e,\widetilde{K}v)_n^2}{\lla v\rra^2_{\alpha}}\right]
    &\le\frac 1{n^2}\sum^n_{j=1}(1+\alpha\lambda_j)^{-1}\mathbb{E}\left(\sum^n_{i=1}e(s_i )(\widetilde{K}\psi_j)(x_i)\right)^2 \\
    &\leq \frac{\sigma^2}{n}\sum^n_{j=1}\frac{1}{1+\alpha\lambda_j},
\end{align*}
where we have used $\|\widetilde{K}\psi_j\|_n=1$, $E[e(s_i)e(s_j)]=\delta_{i,j}$ and $E[e(s_i)^2] \le \sigma^2$. 

It follows from Lemma \ref{eigenvalue-problem} that 
\[\sum^n_{j=1}\frac{1}{1+\alpha\lambda_j} \le C\sum^n_{j=1}\frac{1}{1+\alpha j^{2m}}.\] 
Moreover, 
\[\sum^n_{j=1}\frac{1}{1+\alpha j^{2m}} \le \int^\infty_{1}(1+\alpha t^{2m})^{-1}dt = \alpha^{-\frac{1}{2m}} \int^\infty_{\alpha^{\frac{1}{2m}}} \frac{1}{1+t^{2m}}dt \le C\alpha^{-\frac{1}{2m}},\] 
where using the fact that the improper integral 
$$\int^\infty_{\alpha^{\frac{1}{2m}}} \frac{1}{1+t^{2m}}dt$$
is convergent for the positive integer \(m.\)
Therefore,
\[\mathbb{E}\left[\sup_{v\in V_n}\frac{(e,\widetilde{K}v)_n^2}{\lla v\rra^2_{\alpha}}\right]
    \le C\frac{\sigma^2}{n\alpha^{\frac{1}{2m}}}.\]
This, along with (\ref{estimate under energy norm}), (\ref{energy-estimate}) and (\ref{identity}), completes the proof.

\end{proof}
\begin{rem}
    Theorem \ref{boundedness_x_n} suggests that an optimal choice for the parameter \(\alpha\) satisfies the condition  
\begin{equation}\label{optimal_para}
    \alpha^{\frac{1}{2}+\frac{1}{4m}} = O\bigl(\sigma n^{-\frac{1}{2}}\|x^\dagger\|^{-1}\bigr).
\end{equation}

\end{rem}

\begin{cor}\label{E_Kx}
    If the regularization parameter  \(\alpha\) are chosen to satisfy the a priori parameter choice (\ref{optimal_para}), then the optimal upper bound for convergence rate is given by
    \[\mathbb{E} \big[|Kx_{n,\alpha} - Kx^\dagger|^2_n\big] = O((\sigma n^{-\frac{1}{2}})^{\frac{4m}{1+2m}}\|x^\dagger\|^{\frac{2}{1+2m}}).\]
\end{cor}

\subsubsection{Weak convergence in the dual space}
It is worth noting that Theorem \ref{boundedness_x_n} establishes only that the error $x_{n,\alpha}-x^\dagger$ in the $L^2$-norm remains bounded in expectation. We will demonstrate that, under a weaker topology, the error $x_{n,\alpha}-x^\dagger$ actually converges in expectation.

By the singular system $(s_j; \phi_j, \psi_j)$ of the compact linear operator $ K:L^2(a, b) \to L^2(a, b) $, we know that the $\{s_j^2\}_{j=1}^\infty$ are the non-zero eigenvalues of the self-adjoint operator $K^*K: L^2(a, b) \to L^2(a, b)$ and $\{\phi_j\}_{j=1}^{\infty}$ are a corresponding complete orthonormal system of eigenvectors of $K^*K$, which spans $\overline{R(K^*)}=\overline{R(K^*K)}=N(K)^\bot$. Define a subspace of $L^2(a,b)$
\begin{equation}\label{def_W}
  W = \left\{ f \in N(K)^\bot : \sum_{j=1}^\infty \frac{f_j^2}{s_j}  < \infty, f_j = (f, \phi_j) \right\}
\end{equation}
with the norm \(\|f\|_W := \left( \sum_{j=1}^\infty \frac{f_j^2}{s_j}  \right)^{1/2} \), induced by \((f,g)_W :=  \sum_{j=1}^\infty \frac{f_j g_j}{s_j}.\)
\begin{theorem}\label{W_spectral_rep}
   Let $W$ be defined as above. Then $W$ is a Hilbert space, and moreover, it coincides with the range of the fractional operator $(K^*K)^{1/4}$; that is,
$$
W = R\left[(K^*K)^{1/4}\right].
$$
\end{theorem}
\begin{proof}
Define the mapping
$$
\Phi:\;W\longrightarrow\ell^2,
\qquad
f=\sum_{j=1}^\infty f_j\,\phi_j\;\mapsto\;\bigl(f_1/s_1^{1/2},\,f_2/s_2^{1/2},\,\dots\bigr).
$$
For any $f,g\in W$, we have
$$
(f,g)_W
=\sum_{j=1}^\infty\frac{f_j\,g_j}{s_j}
=\sum_{j=1}^\infty\Bigl(\frac{f_j}{s_j^{1/2}}\Bigr)\Bigl(\frac{g_j}{s_j^{1/2}}\Bigr)
=\langle\Phi(f),\,\Phi(g)\rangle_{\ell^2},
$$
so that $\Phi$ is an isometric isomorphism from $W$ onto the closed subspace $\Phi(W)\subseteq\ell^2$.  Since $\ell^2$ is complete and closed subspaces of complete spaces are complete, $W$ is complete under the inner product $(\cdot,\cdot)_W$.  Hence $(W,(\cdot,\cdot)_W)$ is a Hilbert space.

The self-adjoint operator $K^*K$ have spectral decomposition $(K^*K)\phi_j = s_j^2\phi_j, \quad j=1,2,\dots$. Then for any $g\in L^2(a,b)$,
$$
(K^*K)^{1/4}g = \sum_{n=1}^\infty s_j^{1/2}\,(g,\phi_j)\,\phi_j.
$$
If $f\in R((K^*K)^{1/4})$, there exists $g\in L^2(a,b)$ such that
$$
    f = (K^*K)^{1/4}g = \sum_{n=1}^\infty s_j^{1/2}g_j\,\phi_j,
    \quad g_j=(g,\phi_j).
$$
Hence $f_j = s_j^{1/2}g_j$ and
  $$
    \|f\|_W^2 = \sum_{n=1}^\infty \frac{f_j^2}{s_j} = \sum_{n=1}^\infty g_j^2 = \|g\|_{L^2}^2 < \infty.
  $$
Thus $f\in W$, giving $R((K^*K)^{1/4})) \subseteq W.$

Conversely, if $f\in W$, let $f_j=(f,\phi_j)$ and define $ g := \sum_{j=1}^\infty \frac{f_j}{s_j^{1/2}}\,\phi_j.$ Since 
$$
\sum_{j=1}^\infty \left(\frac{f_j}{s_j^{1/2}}\right)^2 = \|f\|_W^2 < \infty,
$$ we have $g\in L^2(a,b)$.  Then
  $$
    (K^*K)^{1/4}g = \sum_{j=1}^\infty s_j^{1/2}\frac{f_j}{s_j^{1/2}}\,\phi_j = \sum_{j=1}^\infty f_j\,\phi_j = f.
  $$
Hence $f\in R((K^*K)^{1/4}))$, so $W \subseteq R((K^*K)^{1/4})).$

Combining the two inclusions yields $ W = R((K^*K)^{1/4})$, which completes the proof. 

\end{proof}
\begin{lemma}\label{s_decay_dual}
Let \( W \) be the subspace of \( L^2(a,b) \) defined as (\ref{def_W}). Then, identifying linear functionals on $W$ with elements of $L^2(a,b)$ via the $L^2$-inner product, the dual space $W^*$ is
\begin{equation}\label{exp_dual_w}
    W^*=\Bigl\{u\in N(K)^\bot:\sum_{j=1}^\infty s_j\,u_j^2<\infty, u_j =(u,\phi_j)  \Bigr\},
\end{equation}
equipped with the norm $\|u\|_{W^*}= \sup_{0\ne v\in W}\frac{|(u,v)_{L^2}|}{\|v\|_W} =\Big(\sum_{j=1}^\infty s_j\,u_j^2\Big)^{1/2},$
and the duality pairing is $\langle u,v\rangle=\,(u,v)_{L^2}$. 

\end{lemma}

\begin{proof}
Denote
\[
S := \left\{ u\in N(K)^\bot : \sum_{j=1}^\infty s_j u_j^2 < \infty \right\},\qquad
\|u\|_S^2 := \sum_{j=1}^\infty s_j u_j^2,
\]
where \(u_j = (u, \phi_j)_{L^2}\). For \(u \in N(K)^\bot\), define the linear functional \(L_u : W \to \mathbb{R}\) by \(L_u(v) = (u, v)_{L^2}\).
We divide our proof in four steps

 \textbf{Step 1}: If \(u \in S\), then \(L_u\) is bounded on \(W\). 
 Take any \(u \in S\) and \(v \in W\). Expand the \(L^2\) inner product:
\[
|(u, v)_{L^2}| = \left| \sum_{j=1}^\infty u_j v_j \right|
= \left| \sum_{j=1}^\infty (\sqrt{s_j} u_j) \left( \frac{v_j}{\sqrt{s_j}} \right) \right|.
\]
Applying the Cauchy–Schwarz inequality, we obtain
\[
|(u, v)_{L^2}| \le \left( \sum_{j=1}^\infty s_j u_j^2 \right)^{1/2} \left( \sum_{j=1}^\infty \frac{v_j^2}{s_j} \right)^{1/2} = \|u\|_S \|v\|_W.
\]
Thus, \(L_u\) is a bounded linear functional, and \(\|L_u\|_{W \to \mathbb{R}} \le \|u\|_S\).

 \textbf{Step 2}: If \(L_u\) is bounded on \(W\), then \(u \in S\). Suppose \(u \in N(K)^\bot\) and \(L_u\) is bounded on \(W\), i.e., \(\|L_u\|_{W \to \mathbb{R}} < M < \infty\). For each \(N \in \mathbb{N}\), consider the finite vector \(v^{(N)} := \sum_{j=1}^N s_j u_j\, \phi_j\). Then \(v^{(N)} \in W\), since 
 \[\|v^{(N)}\|_W^2 = \sum_{j=1}^N (s_j u_j)^2 / s_j = \sum_{j=1}^N s_j u_j^2 < \infty.\]
Moreover, \((u, v^{(N)})_{L^2} = \sum_{j=1}^N u_j (s_j u_j) = \sum_{j=1}^N s_j u_j^2 = \|v^{(N)}\|_W^2.\)

By boundedness,
\[
\|v^{(N)}\|_W^2 = |(u, v^{(N)})_{L^2}| \le M \|v^{(N)}\|_W.
\]
Let \(A_N := \|v^{(N)}\|_W^2 \ge 0\), then \(A_N \le M \sqrt{A_N}\). If \(A_N > 0\), dividing both sides by \(\sqrt{A_N}\) yields \(\sqrt{A_N} \le M\), and thus \(A_N \le M^2\); if \(A_N = 0\), then clearly \(A_N \le M^2\). Therefore, for all \(N\), 
\[\sum_{j=1}^N s_j u_j^2 = A_N \le M^2.\]
Letting \(N \to \infty\), we obtain \(\sum_{j=1}^\infty s_j u_j^2 \le M^2 < \infty\), i.e., \(u \in S\). 

 \textbf{Step 3}: \(\|L_u\|_{W \to \mathbb{R}} = \|u\|_S\). 
 From Step 1, \(\|L_u\|_{W \to \mathbb{R}} \le \|u\|_S\). For the reverse inequality, take a nonzero \(u \in S\) since the case \(u \equiv 0\) is clear. Define
\(v := \sum_{j=1}^\infty s_j u_j\, \phi_j.\)
Since \(\sum_{j=1}^\infty s_j u_j^2 < \infty\), we have
\[
\|v\|_W^2 = \sum_{j=1}^\infty \frac{(s_j u_j)^2}{s_j} = \sum_{j=1}^\infty s_j u_j^2 < \infty,
\]
so \(v \in W\). Moreover,
\[
(u, v)_{L^2} = \sum_{j=1}^\infty u_j (s_j u_j) = \sum_{j=1}^\infty s_j u_j^2 = \|v\|_W^2.
\]
Let \(v_1 := v / \|v\|_W\), then
\[
|L_u(v_1)| = |(u, v_1)_{L^2}| = \frac{|(u, v)_{L^2}|}{\|v\|_W} = \frac{\|v\|_W^2}{\|v\|_W} = \|v\|_W = \|u\|_S.
\]
Thus, \(\|L_u\|_{W \to \mathbb{R}} \ge \|u\|_S\). Combined with the upper bound, we have \(\|L_u\|_{W \to \mathbb{R}} = \|u\|_S\). 

\textbf{Step 4}: \(S \) and \( W^*\) are isometric.
According to Theorem \ref{W_spectral_rep}, \(W\) is a Hilbert space. By the Riesz representation theorem, for any bounded linear functional \(L \in W^*\), there exists a unique \(z \in W\) such that
\[
L(v) = \langle z, v \rangle_W = \sum_{j=1}^\infty \frac{z_j v_j}{s_j} \qquad (\forall v \in W).
\]
Define \(u\) by \(u_j := z_j / s_j\) with respect to the basis \(\{\phi_j\}\). Since \(z \in W\), we have
\[
\sum_{j=1}^\infty s_j u_j^2 = \sum_{j=1}^\infty \frac{z_j^2}{s_j} = \|z\|_W^2 < \infty,
\]
so \(u \in S\). Moreover, for any \(v \in W\),
\[
(u, v)_{L^2} = \sum_{j=1}^\infty u_j v_j
= \sum_{j=1}^\infty \frac{z_j}{s_j} v_j
= \langle z, v \rangle_W = L(v).
\]
Therefore, every \(L \in W^*\) can be written as \(L = L_u\) and the norm identity from the previous step shows \(\|L\|_{W^*} = \|u\|_S\). Conversely, Step 2 shows that each \(u \in S\) defines a bounded functional \(L_u\) on \(W\). Thus, \(u \mapsto L \) is an isometric isomorphism from \(S \) onto \( W^*\). 

\end{proof} 

\begin{theorem}
     Let $\alpha \geq h^{2m}$. Then
  \[
    \mathbb{E}\left[\|x_{n,\alpha} - x^\dagger\|_{W^*}^2\right] \leq C\alpha^{\frac{1}{2}}\|x^\dagger\|^2 + C \frac{C\sigma^2}{n\alpha^{\frac{1}{2}+\frac{1}{1+2m}}}.
  \]
\end{theorem}
\begin{proof}
For any \( v \in N(K)^\perp \), we have the spectral expansion
\(v = \sum_{j=1}^\infty v_j \phi_j \) with \(v_j = (v, \phi_j),\)
which yields $\|v\|^2 = \sum_{j=1}^\infty v_j^2$ and $\|K v\|^2 = \sum_{j=1}^\infty s_j^{2} v_j^2$. 

Similarly, for any $g \in N(K)^\bot$ represented as $g = \sum_{j=1}^\infty g_j \phi_j$, with $g_j = (g, \phi_j)$, Lemma~\ref{s_decay_dual} together with the Cauchy–Schwarz inequality implies that
\[
  \|v\|_{W^*}  = \Big(\sum_{j=1}^\infty s_j\,v_j \cdot v_j\Big)^{1/2} \leq \|K v\|^{1/2} \|v\|^{1/2}.
\]
Substituting $v = x_{n,\alpha} - x^\dagger$ in the above inequality gives
\begin{equation}\label{ineq}
  \|x_{n,\alpha} - x^\dagger\|_{W^*}^2 \leq \|K x_{n,\alpha} - Kx^\dagger\| \|x_{n,\alpha} - x^\dagger\|.
\end{equation}
By (\ref{relationship between the norms}), we have
\begin{align*}
    \| Kx_{n,\alpha} - Kx^\dagger \| &\leq  C\big(\| Kx_{n,\alpha} -Kx^\dagger \|_n^2 + h^{2m}\|K(x_{n,\alpha} -x^\dagger)\|_{H^m(a,b)}^2\big)^\frac{1}{2}\\
    & \leq C\big(\| Kx_{n,\alpha} -Kx^\dagger \|_n^2 + h^{2m}\|x_{n,\alpha} -x^\dagger\|^2\big)^\frac{1}{2} \\
    & \leq C\big(\|Kx_{n,\alpha} -Kx^\dagger\|_n^2 + \alpha\|x_{n,\alpha} -x^\dagger\|^2\big)^\frac{1}{2} \\
    & \leq C\big(\|Kx_{n,\alpha} -Kx^\dagger\|_n + \alpha^\frac{1}{2}\|x_{n,\alpha} -x^\dagger\|\big).
\end{align*}
where we used the boundedness of the operator \(K : L^2(a,b) \to H^m(a,b)\) and \(\alpha \geq h^{2m}\),
 
Then, we can derive from (\ref{ineq}) that
\begin{align}
    \|x_{n,\alpha} - x^\dagger\|_{W^*}^2 &\leq C\big(|Kx_{n,\alpha} -Kx^\dagger|_n + \alpha^\frac{1}{2}\|x_{n,\alpha} -x^\dagger\|\big)\|x_{n,\alpha} - x^\dagger\| \nonumber  \\
    &\leq C\alpha^{-\frac{1}{4}}|Kx_{n,\alpha} -Kx^\dagger|_n\cdot\alpha^{\frac{1}{4}}\|x_{n,\alpha} - x^\dagger\| + C\alpha^{\frac{1}{2}}\|x_{n,\alpha} - x^\dagger\|^2 \nonumber \\
    & \leq C\alpha^{-\frac{1}{2}}|Kx_{n,\alpha} -Kx^\dagger|_n^2 + C\alpha^{\frac{1}{2}}\|x_{n,\alpha} - x^\dagger\| + C\alpha^{\frac{1}{2}}\|x_{n,\alpha} - x^\dagger\|^2  \nonumber\\
    & = C\alpha^{-\frac{1}{2}}|Kx_{n,\alpha} -Kx^\dagger|_n^2 + C\alpha^{\frac{1}{2}}\|x_{n,\alpha} - x^\dagger\|^2.  \label{eq:final}
\end{align}
 
Applying inequalities (\ref{ineq1}) and (\ref{ineq2}) to bound the two terms on the right-hand side, we conclude that
\begin{align*}
    \mathbb{E}\big[\|x_{n,\alpha} - x^\dagger\|_{W^*}^2\big] &\leq C\alpha^{-\frac{1}{2}}\mathbb{E}\big[|Kx_{n,\alpha} -Kx^\dagger|_n^2\big] + C\alpha^{\frac{1}{2}}\mathbb{E}\big[\|x_{n,\alpha} - x^\dagger\|^2\big] \\
    & \le C\alpha^{\frac{1}{2}}\|x^\dagger\|^2 + C \frac{\sigma^2}{n\alpha^{\frac{1}{2}+\frac{1}{2m}}}.
\end{align*}

\end{proof}

\begin{cor}\label{E_x_n}
    If the regularization parameter  \(\alpha\) are chosen to satisfy the a priori parameter choice (\ref{optimal_para}), then the optimal upper bound for convergence rate is given by
    \[ \mathbb{E}\left[\|x_{n,\alpha} - x^\dagger\|_{W^*}^2\right] = O((\sigma n^{-\frac{1}{2}})^{\frac{2m}{1+2m}}\|x^\dagger\|^\frac{2+2m}{1+2m}).\]
\end{cor}

\section{ Convergence for data with sub-Gaussian noise}\label{sec4}

In this section, we consider the integral operator \(K: L^2(a,b) \to L^2(a,b)\) given by (\ref{integral_eq}), where the kernel \(k(s,t)\in L^2((a,b)^2)\) and the range of \(K\) lies within the Sobolev space \(H^m(a,b)\) for some positive integer \(m\). Assume the noises $e(s_1)$, $e(s_2)$, $\cdots$, $e(s_n)$ are independent and identically distributed sub-Gaussian random variables with parameter $\sigma>0$. Under the above sub-Gaussian random noises, we investigate the probabilistic convergence behavior of the empirical error $|Kx_{n,\alpha}-K x^{\dagger}|_n$ and the weak error $\|x_{n,\alpha}-x^\dagger\|_{W^*}$, which characterizes the tail property of \(\mathbb{P}\big(|Kx_{n,\alpha} - Kx^\dagger|_n \ge z\big)\) and \(\mathbb{P}\big(|x_{n,\alpha} - x^\dagger|_{W^*} \ge z\big)\).

\subsection{Preliminaries}

Our analysis relies on several foundational results from empirical process theory; see \cite{Geer2000, Vaart1996} for more details. As a starting point, we recall the definition of sub-Gaussian random variables, which play a central role in quantifying concentration phenomena throughout this framework.
\begin{definition}
    A random variable $Z$ with mean $\mu = \mathbb{E}[Z]$ is \textit{sub-Gaussian} if there is a positive number $\sigma$ such that
    \begin{equation}\label{def_of_subgauss}
        \mathbb{E}[\exp(\lambda (Z - \mu))] \leq \exp(\sigma^2 \lambda^2 / 2) \quad \text{for all } \lambda \in \mathbb{R}. 
    \end{equation}
\end{definition}
The constant $\sigma$ is referred to as the \textit{sub-Gaussian parameter} of $Z$.  In particular, any normal random variable of variance $\sigma^2$ satisfies (\ref{def_of_subgauss}) and hence is sub‐Gaussian with parameter $\sigma$.  

\begin{definition}
Let $\psi : \mathbb{R}_+ \to \mathbb{R}_+$ be a strictly increasing convex function that satisfies $\psi(0) = 0$. The $\psi$-Orlicz norm of a random variable $Z$ is defined as
\[
\|Z\|_\psi := \inf \{ t > 0 \mid \mathbb{E}[\psi(t^{-1} |Z|)] \leq 1 \},
\]
where $\|Z\|_\psi$ is considered infinite if there is no finite $t$ for which the expectation $\mathbb{E}[\psi(t^{-1} |Z|)]$ exists. 
\end{definition}

The following two lemmas establish the equivalence between the finiteness of $\|Z\|_{\psi_p}$ and the presence of exponential tail decay.
\begin{lemma}
    If $\|Z\|_{\psi_p} < +\infty$, then 
    \begin{equation}\label{con_ineq}
        \mathbb{P}(|Z| > t) \leq 2 \exp(-\frac{ t^p}{\|Z\|_{\psi_p}^{p}}) \quad \text{for all } t > 0.
    \end{equation}
\end{lemma}
\begin{lemma}\label{e_tail}
If a random variable $Z$ satisfies the exponential‐type tail bound
    \[
    P(|Z| > z) \leq C_1 e^{-C_2x^p}
    \]
    where $C_1, C_2 > 0$ are constants and $p \ge 1$. Then its Orlicz norm satisfies
    $$
    \|Z\|_{\psi_p} \;\le\; \left( \frac{1 + C_1}{C_2} \right)^{1/p}.
    $$
\end{lemma}

For the subsequent analysis, we will use the following result, which provides an improvement over Lemma \ref{e_tail} in the case $p=2$ (see \cite{Chen2018}).
\begin{lemma}\label{tail-p}
    Let $Z$ be a random variable. If there exist constants $C_1, C_2>0$ such that, for every scale parameter $\alpha>0$ and all $z\ge1$,
$$
\mathbb{P}\bigl(|Z|>\alpha\,(1+z)\bigr)\;\le\;C_1\exp\!\bigl(-z^2/C_2^2\bigr)\,.
$$
Then
$$\|Z\|_{\psi_2}\le C(C_1,C_2)\,\alpha$$ 
where the constant $C(C_1,C_2)$ dependes only on $C_1$ and $C_2$.
\end{lemma}

 A random process $\{Z_t : t \in T\}$ is said to be \textit{sub-Gaussian} if the $\psi_2$-norms of the increments $Z_s - Z_t$ are finite. Equivalently, there exists a semimetric $d$ on $T$ such that for all $x>0$ and every pair $s,t\in T$,
$$
\mathbb{P}\bigl(|Z_s - Z_t|>x\bigr)\;\le\;2\exp\, \Bigl(-\tfrac{x^2}{2\,d(s,t)^2}\Bigr).
$$
The following useful results can be found in \cite{Chen2018}.
\begin{lemma}\label{K_n sub-Gaussian}
$\{E_n(u):=(e,\widetilde{K}u)_n: u\in N(K)^\bot\}$ is a sub-Gaussian random process with respect to the semi-distance $\sd(u,v)=
\sigma n^{-1/2}|\widetilde{K}u-\widetilde{K}v|_n$ for $u, v\in N(K)^\bot$, where $\widetilde{K}$ is defined as Lemma \ref{min-problem}. 
\end{lemma}

To present the main results in an informal manner, we introduce also the notions of covering number and covering entropy. Let $\mathcal{F}$ be a collection of real‐valued functions on domain $\mathcal{X}$, endowed with a (semi)norm $\|\cdot\|$.

\begin{definition}
Fix $\varepsilon>0$.  An $\varepsilon$\emph{-cover} of $\mathcal{F}$ consists of functions $h_1,\dots,h_m$ (not necessarily in $\mathcal{F}$, but satisfying $\|h_i\|<\infty$) such that every $f\in\mathcal{F}$ lies within distance $\varepsilon$ of at least one center:
$$
\mathcal{F}\;\subset\;\bigcup_{i=1}^m\{\,g:\|g-h_i\|<\varepsilon\}\,. 
$$
The $\varepsilon$\emph{-covering number} $N(\varepsilon,\mathcal{F},\|\cdot\|)$ is the smallest $m$ for which such a cover exists.  Its logarithm,
$\log N(\varepsilon,\mathcal{F},\|\cdot\|),$
is referred to as the (metric) entropy of $\mathcal{F}$ at scale $\varepsilon$.
\end{definition}

An essential element of our analysis is the maximal inequality presented in \cite[Section 2.2.1]{Vaart1996}.
\begin{lemma}\label{maximal_inequality}
If $\{Z_t:t\in T\}$ is a separable sub-Gaussian random process, then it holds for some constant $C>0$ that 
\begin{equation*}
    \|\sup_{s,t\in T}|Z_s-Z_t|\|_{\psi_2}\le C\int^{\diam\, T}_0\sqrt{\log N\Big(\frac\vep 2,T,\sd\Big)}\ d\vep\,.
\end{equation*}
\end{lemma}

Birman and Solomyak’s foundational work \cite{Birman1967} yields the following estimate for the metric entropy of Sobolev spheres.
\begin{lemma}\label{covering entropy of Soblev space}
Let $Q=[0,1]^d\subset\mathbb{R}^d$, and denote by
$$
SW^{\alpha,p}(Q)
\;=\;
\bigl\{\,f\in W^{\alpha,p}(Q): \|f\|_{W^{\alpha,p}}\le1\bigr\}
$$
the unit sphere in the Sobolev space $W^{\alpha,p}(Q)$ with smoothness index $\alpha>0$ and integrability exponent $p\ge1$.  Then, for all sufficiently small $\varepsilon>0$, there exists a constant $C$ (depending only on $d,\alpha,p$) such that
$$
\log\mathcal{N}\bigl(\varepsilon,SW^{\alpha,p}(Q),\|\cdot\|_{L^q(Q)}\bigr)
\;\le\;
C\,\varepsilon^{-d/\alpha}.
$$
Here the range of $q$ is determined by the Sobolev embedding: if $\alpha p>d$, any $1\le q\le\infty$ is admissible; whereas when $\alpha p\le d$, one must restrict to $1\le q < q^*$, where $q^*\;=\;\frac{p}{\,1-\alpha p/d\,}\,.$
\end{lemma}

\subsection{Main results with sub-Gaussian noise}

For any $\rho > 0$, define the subset
$$
F_{\rho} := \left\{ x \in N(K)^\perp : \|x\| \leq \rho \right\}.
$$
We are interested in quantifying the metric complexity of the image $K(F_\rho)$ under the operator $K$, measured with respect to the $L^\infty(a,b)$-norm. The following lemma provides an upper bound on the covering entropy of this image set.

\begin{lemma}\label{covering entropy}
There exists a constant $C > 0$, independent of $\rho$ and $\varepsilon$, such that for all $\varepsilon > 0$,
$$
\log \mathcal{N}\left(\varepsilon,\, K(F_{\rho}),\, \|\cdot\|_{L^\infty(a,b)}\right)
\;\leq\;
C \left( \frac{\rho}{\varepsilon} \right)^{1/m}.
$$
\end{lemma}

\begin{proof}
    Since \(K\) is a bounded linear operator from $L^2(a,b)$ to $H^m(a,b)$, there exists a constant \(C > 0\) such that 
    \(\|Kx\|_{H^m(a, b)} \le C \|x\|.\)
    Then \(\|Kx\|_{H^m(a, b)} \le C\rho \) for any $x \in F_{\rho}$. Using Lemma \ref{covering entropy of Soblev space}, we can derive the conclusion.
    
\end{proof}

We are now in a position to state the main result of this section.
\begin{theorem}\label{Main_th3}
    Let $x_{n,\alpha}$ and $x^\dagger$ be the regularized solution and best-approximate solution of equation (\ref{integral_eq}) respectively. Denote by $\rho_0=\|x^\dagger\| + \sigma n^{-\frac{1}{2}}$. With the a priori choice \( \alpha^{\frac{1}{2} + \frac{1}{4m}} = O(\sigma n^{-\frac{1}{2}} \rho_0^{-1}) \), we have
\begin{equation}\label{H_ex_tail}
    \mathbb{P}(|Kx_{n,\alpha}-Kx^\dagger|_n\ge \alpha^{\frac{1}{2}}\rho_0z)\le 2\,e^{-Cz^2} 
\quad {\rm and} \q \mathbb{P}(\|x_{n,\alpha}\|\ge \rho_0z)\le 2\,e^{-Cz^2}.
\end{equation}
where \(C >0\) is a constant.
\end{theorem}

\begin{proof}
    By the minimizing property of \(x_{n,\alpha}\) in (\ref{Tikhonov reg}), we have
    \[ |Kx_{n,\alpha}-\boldsymbol w|^2_n + \alpha \|x_{n,\alpha}\|^2 \le |Kx^\dagger-\boldsymbol w|^2_n + \alpha \|x^\dagger\|^2 
         = |e|_n^2 + \alpha \|x^\dagger\|^2, \]
    and \(|Kx_{n,\alpha}-\boldsymbol w|^2_n = |Kx_{n,\alpha} - Kx^\dagger|_n^2 -2(Kx_{n,\alpha} - Kx^\dagger,e)_n + |e|_n^2\). Thus,    
    \begin{equation}\label{Optimality-reg}
        |Kx_{n,\alpha}-Kx^\dagger|^2_n + \alpha \|x_{n,\alpha}\|^2 \leq 2(e,Kx_{n,\alpha}-Kx^\dagger)_n + \alpha \|x^\dagger\|^2.
    \end{equation}
    Let $\theta >0,\ \rho>0$ be two constants to be determined later, and we set for $i,j\ge 1$, 
    \begin{equation*}
         A_0=[0,\theta), \q A_i=[2^{i-1}\theta,2 ^i\theta), \q B_0=[0,\rho), \q B_j=[2^{j-1}\rho,2^j\rho)\,.
    \end{equation*}
    For $i,j\ge 0$, we further define
    \begin{equation*}
        F_{ij}= \{x \in N(K)^\bot:~  |\widetilde{K}x|_n \in A_i ~,~ \|x\| \in B_j \},
    \end{equation*}
    where $\widetilde{K}$ is defined as Lemma \ref{min-problem}.
    Then we can readily see 
    \begin{equation}\label{P_Pro_ineq}
        \mathbb{P}(|Kx_{n,\alpha}-Kx^\dagger|_n>\theta)=\mathbb{P}(|\widetilde{K}x_{n,\alpha}-\widetilde{K}x^\dagger|_n>\theta)\le\sum_{i=1}^\infty\sum_{j=0}^\infty \mathbb{P}(x_{n,\alpha}-x^\dagger \in F_{ij}).
    \end{equation}
    Now we estimate $\mathbb{P}(x_{n,\alpha}-x^\dagger \in F_{ij})$ for each pair $\{i, j\}$.
    By Lemma \ref{K_n sub-Gaussian}, we know 
$\{(e,\widetilde{K}v)_n:v\in N(K)^\bot\}$ is a sub-Gaussian random process with respect to the semi-distance $\sd(u,v)=\sigma n^{-1/2}|\widetilde{K}u-\widetilde{K}v|_n$. 

    Using the semi-distance, we have \(\diam(F_{ij})\le  \sigma n^{-1/2}\cdot 2^{i+1}\theta.\) 
In fact, 
\begin{align*}
    \diam(F_{ij})  &= \sup_{u,v\in F_{i,j}} d(u, v) = \sup_{u,v\in F_{i,j}} \sigma n ^{-\frac{1}{2}}|\widetilde{K}u-\widetilde{K}v|_n \\
    &\le \sup_{u,v\in F_{i,j}} \sigma n ^{-\frac{1}{2}}(|\widetilde{K}u|_n +|\widetilde{K}v|_n) \\
    &\le \sigma n ^{-\frac{1}{2}}(2^i\theta +2^i\theta).
\end{align*}
Then we can deduce by using Lemma \ref{maximal_inequality} that 
\begin{align*}
    \|\sup_{x-x^\dagger\in F_{ij}}|(e,\widetilde{K}x-\widetilde{K}x^\dagger)_n|\|_{\psi_2}&\le C\int^{\sigma n^{-1/2}\cdot 2^{i+1}\theta}_0\sqrt{\log N\left(\frac\vep 2,F_{ij}, \sd\right)}\,d\vep\nn\\
    &=C\int^{\sigma n^{-1/2}\cdot 2^{i+1}\theta}_0\sqrt{\log N\left(\frac\vep{2\sigma n^{-1/2}},F_{ij}, |\widetilde{K}\cdot|_n\right)}\,d\vep.
\end{align*}
    By Lemma \ref{covering entropy}, we have the estimate for the covering entropy
    \begin{align*}
        \log N\left(\frac\vep{2\sigma n^{-\frac{1}{2}}},F_{ij}, |\widetilde{K}\cdot|_n\right) &\le \log N(\frac\vep{2\sigma n^{-\frac{1}{2}}},F_{ij}, \|\widetilde{K}\cdot\|_{L^\infty(a,b)})\\
        & = \log N(\frac\vep{2\sigma n^{-\frac{1}{2}}},\widetilde{K}(F_{ij}), \|\cdot\|_{L^\infty(a,b)}) \\
        &\le C\left(\frac{2\sigma n^{-\frac{1}{2}}\cdot2^j\rho}{\vep}\right)^{1/m},
    \end{align*}
    where the equality holds because \(\widetilde{K}\) is injective.
Using this, we can further derive
\begin{equation*}
    \begin{aligned}
    \|\sup_{x - x^\dagger \in F_{ij}} \, |(e, \widetilde{K} x - \widetilde{K} x^\dagger)_n|\, \|_{\psi_2} & \leq C \int_0^{\sigma n^{-\frac{1}{2}} 2^{i+1} \theta} \left( \frac{2 \sigma n^{-\frac{1}{2}} \cdot 2^j \rho}{\vep} \right)^{\frac{1}{2m}} \, d\varepsilon \\
    & = C \sigma n^{-\frac{1}{2}} (2^j \rho)^{\frac{1}{2m}} (2^i \theta)^{1 - \frac{1}{2m}}.
\end{aligned}
\end{equation*}
Then by using the estimates (\ref{Optimality-reg}) and (\ref{con_ineq}), we have for \( i, j \geq 1 \),
\begin{align*}
    \mathbb{P}(x_{n,\alpha} - x^\dagger \in F_{ij}) &\leq \mathbb{P}\left( 2^{2(i-1)} \theta^2 + \alpha \cdot 2^{2(j-1)} \rho^2 \leq 2 \sup_{x - x^\dagger \in F_{ij}} |(e, \widetilde{K}x - \widetilde{K}x^\dagger)_n| + \alpha \rho_0^2 \right) \\
    & =\mathbb{P}\left( 2 \sup_{x - x^\dagger \in F_{ij}} |(e, \widetilde{K} x - \widetilde{K} x^\dagger)_n| \ge 2^{2(i-1)} \theta^2 + \alpha \cdot 2^{2(j-1)} \rho^2 - \alpha \rho_0^2 \right)\\
    &\le 2\exp\left[-\left( \frac{2^{2(i-1)} \theta^2 + \alpha \cdot 2^{2(j-1)} \rho^2 - \alpha \rho_0^2}{(2^i)^{1 - \frac{1}{2m}} (2^j)^{\frac{1}{2m}}\cdot \sigma n^{-\frac{1}{2}}\cdot \theta^{1 - \frac{1}{2m}}\cdot \rho^{\frac{1}{2m}}}  \right)^2\right]:=I_{i,j}.
\end{align*}
Taking \( \theta = \alpha^\frac{1}{2} \rho_0 (1 + z) \) and \( \rho = \rho_0 \) , we obtain
\begin{equation*}
    I_{ij} = 2\exp\left[-\left( \frac{(2^{2(i-1)} (1+z)^2 +  2^{2(j-1)} - 1 ) \cdot \alpha^{\frac{1}{2}+\frac{1}{4m}} \rho_0}{(2^i(1+z))^{1 - \frac{1}{2m}} (2^j)^{\frac{1}{2m}}\cdot \sigma n^{-\frac{1}{2}}}  \right)^2\right].
\end{equation*}
 With the a priori choice \( \alpha^{\frac{1}{2} + \frac{1}{4m}} = O(\sigma n^{-1/2} \rho_0^{-1}) \), we have
\begin{equation}\label{Pro_ineq}
    I_{ij} \leq 2 \exp\left[ -C \left( \frac{2^{2(i-1)} z (1 + z) + 2^{2(j-1)}}{(2^i (1 + z))^{1 - \frac{1}{2m}} (2^j)^{\frac{1}{2m}}} \right)^2 \right].
\end{equation}

By Young's inequality that \( ab \leq \frac{a^p}{p} + \frac{b^q}{q} \) for any \( a, b > 0 \) and \( p, q > 1 \) with \( p^{-1} + q^{-1} = 1 \), we obtain
\begin{equation*}
    (2^i (1 + z))^{1 - \frac{1}{2m}} (2^j)^{\frac{1}{2m}} \le C(2^i(1 + z) + 2^j).
\end{equation*}
It follow from (\ref{Pro_ineq}) that for \( i, j \geq 1 \),
\[\mathbb{P}(x_{n,\alpha} - x^\dagger \in F_{ij}) \leq 2 \exp\left[ -C (2^{2i} z^2 + 2^{2j}) \right].\]

Similarly, one can show for \( i \geq 1, j = 0 \) that
\[\mathbb{P}(x_{n,\alpha} - x^\dagger \in F_{i0}) \leq 2 \exp\left[ -C (2^{2i} z^2) \right].\]
Using the facts that 
$$
\sum^\infty_{j=1} \m{exp} \big(-C(2^{2j})\big) \le \m{exp} ({-C})< 1\q \m{and} \q  
\sum^\infty_{i=1} \m{exp} \big({-C(2^{2i}z^2)}\big) \le \m{exp} ({-Cz^2}),$$
along with the above estimates for all $i, j\ge 0$, we obtain finaly that 
\begin{align*}
    \sum_{i=1}^\infty\sum_{j=0}^\infty \mathbb{P}(x_{n,\alpha} - x^\dagger\in F_{ij}) &\le2\sum_{i=1}^\infty\sum_{j=1}^\infty 
\m{exp} ({- C (2^{2i} z^2 + 2^{2j})})+2\sum^\infty_{i=1} \m{exp} ({- C (2^{2i} z^2)})  \\
&\le 4 \m{exp} ({-Cz^2}).
\end{align*}
It follows from (\ref{P_Pro_ineq}) that 
\begin{equation}
    \mathbb{P}(|Kx_{n,\alpha} - Kx^\dagger|_n>\alpha^{\frac{1}{2}}\rho_0(1+z))\le 4 \,\m{exp} ({-Cz^2})\ \ \ \ \forall z\ge 1,
\end{equation}
which implies by means of Lemma \ref{tail-p} that
\begin{equation}\label{first_estimate}
    \|\,|Kx_{n,\alpha} - Kx^\dagger|_n\,\|_{\psi_2} \leq C \alpha^{\frac{1}{2}}\rho_0.
\end{equation}
An argument similar to the estimate (\ref{first_estimate})  by taking \(i \geq 0\) and \(j \geq 1\) in the summation show \begin{equation}\label{second_estimate}
    \|\,\|x_{n,\alpha}\|\,\|_{\psi_2}\le C\rho_0.
\end{equation}
Using (\ref{con_ineq}), we have 
\begin{align*}
     \mathbb{P}(|Kx_{n,\alpha}-Kx^\dagger|_n \ge \alpha^{\frac{1}{2}}\rho_0z) &\le 2\exp(-\frac{\alpha\rho_0z^2}{\|\,|Kx_{n,\alpha} - Kx^\dagger|_n\,\|_{\psi_2}^2})  \\
     & \le 2\exp(-\frac{\alpha\rho_0 ^2z^2}{C \alpha\rho_0^2})\\
     &= 2\exp(-Cz^2)
\end{align*}
Similarly, we obtain
\begin{equation*}
 \mathbb{P}(\|x_{n,\alpha}\|\ge \rho_0z)\le 2\,e^{-Cz^2}.
\end{equation*}
\end{proof}

\begin{cor}
Let \(x_{n,\alpha}\) and \(x^\dagger\) denote the regularized and best-approximation solutions of equation~(\ref{integral_eq}), respectively. Suppose that \(\alpha \ge h^{2m}\). Then the following concentration inequality holds:
\[
\mathbb{P}\!\left(\|x_{n,\alpha} - x^\dagger\|_{W^*} \ge \alpha^{\frac{1}{4}}\rho_0 z\right) \le 2e^{-Cz^2}.
\]
\end{cor}
\begin{proof}
From inequality~(\ref{eq:final}), we have 
\[\|x_{n,\alpha} - x^\dagger\|_{W^*} \le  C\alpha^{-\frac{1}{4}}|Kx_{n,\alpha} -Kx^\dagger|_n + C\alpha^{\frac{1}{4}}\|x_{n,\alpha} - x^\dagger\|.\]
Consequently,
    \[
\|\|x_{n,\alpha} - x^\dagger \|_{W^*} \|_{\psi_2} \leq C \alpha_n^{-\frac{1}{4}} \| ~|Kx_{n,\alpha} -Kx^\dagger|_n \|_{\psi_2} 
+ C \alpha^{\frac{1}{4}} \| \|x_{n,\alpha} - x^\dagger \|\|_{\psi_2}.
\]
Applying estimates~(\ref{first_estimate}) and~(\ref{second_estimate}) yields
\[\|\|x_{n,\alpha} - x^\dagger\|_{W^*}\|_{\psi_2} \le C \alpha^{\frac{1}{4}} \rho_0.\]
Using the concentration inequality~(\ref{con_ineq}), we can derive the conclusion.
\end{proof}

\section{Numerical experiments and discussions}

In this section, we presents numerical experiments designed to validate and exemplify the theoretical results established in Sections \ref{sec3} and \ref{sec4}.

\subsection{Discretization and adaptive parameter choice}

For numerical computations, it is necessary to approximate the spaces \(X = Y = L^2(a,b)\) with finite-dimensional subspaces. For simplicity, we assume throughout this section that the compact operator \( K \) is one-to-one. This assumption does not impose a serious restriction, as the domain \( X \) can always be replaced by the orthogonal complement of the kernel of \( K \), ensuring injectivity. Furthermore, we consider \( V_h \subset L^2(a,b) \) as the finite element space consisting of piecewise linear functions defined on a uniform grid with mesh size \( h = \frac{b-a}{M-1} \), where \( M \) represents the total number of mesh nodes. The finite element discretization of variational form (\ref{variational form}) reads as:
Find \(x_{\alpha,h} \in V_h\) such that 
\begin{equation}\label{eq_v_f}
     (K x_{\alpha,h}, K v_h)_n + \alpha (x_{\alpha,h}, v_h) = (\boldsymbol w, K v_h)_n \quad \forall v_h \in V_h.    
\end{equation}
Here, \(x_{\alpha,h} \in V_h\) amounts to minimizing the following functional
\begin{equation}\label{min_fun_h}
    \min_{x\in V_h}|Kx - \boldsymbol w |_n^2+\alpha \|x\| ^2,
\end{equation}
which is a finite element approximation of the variant of Tikhonov regularization (\ref{min_fun}). An argument similar to the one used in Lemma \ref{min_var_th} shows that the existence and uniqueness of the minimizer \(x_{\alpha, h}\).
Let \( \{\phi_j\}_{j=1}^M \) be a nodal basis of \( V_h \) and \(x_{\alpha,h} =\sum_{j=1}^M c_j\phi_j\) . By this discretization, we obtain the Petrov-Galerkin linear system
\begin{equation}\label{P-G}
  A_M \vec{c}  = \vec{b},
\end{equation}
where
\[
[A_M]_{ij} = (K\phi_j, K\phi_i)_n + \alpha (\phi_j, \phi_i) \quad \text{and} \quad [\vec{b}]_i = (\boldsymbol w, K\phi_i)_n.
\]

The proper choice of the regularization parameter $\alpha$ in the variant of Tikhonov regularization is crucial for balancing fidelity to the observed data against the stability of the solution. To achieve an optimal compromise, one typically resorts to data‐driven selection rules such as the generalized cross‐validation, which minimizes an estimate of the predictive error, or the L‐curve method, in which one identifies the corner of the plot of solution norm versus residual norm. Here, we introduce a self-consistent algorithm for determining the regularization parameter \(\alpha\) guided by the a priori choice 
\begin{equation}\label{priori_parameter}
    \alpha^{\frac{1}{2} + \frac{1}{4m}} = O(\sigma n^{-\frac{1}{2}} (\|x^\dagger\| + \sigma n^{-\frac{1}{2}})^{-1}),
\end{equation}
which is established in Theorem \ref{Main_th3}. Note that this rule (\ref{priori_parameter}) can also serve as an approximate replacement for the prior parameter selection rule (\ref{optimal_para}) in sectoin \ref{sec3}, since they differ only by the additive term $\sigma\,n^{-1/2}$ in the denominator, which is typically negligible for large $n$ in practice.

This parameter‐choice rule, however, presupposes knowledge of the true solution norm $\|x^\dagger\|$ and the noise standard variance $\sigma$, both of which are generally unavailable in applications. To address this, we introduce a self‐consistent strategy that iteratively estimates these quantities: at each step, the unknown solution norm $\|x^\dagger\|$ is approximated by $\|x_{\alpha,h}\|$, and the noise standard variance $\sigma$ is estimated via the data discrepancy $|K\,x_{\alpha,h}-\boldsymbol w|_n$ noting that $|K\,x^\dagger - \boldsymbol w|_n = |e|_n$.  The complete adaptive algorithm is presented in Algorithm \ref{alg:adpa-alpha}.

\begin{algorithm}[hbt!]
  \caption{Adaptive Selection of the Regularization Parameter \(\alpha\)}\label{alg:adpa-alpha}
  \begin{algorithmic}[1]
    \STATE Initial estimate \(\alpha^{(0)}\), maximum iterations \(p\)
    \FOR{\(k = 0,\ldots, p-1\)}
      \STATE Solve the  finite element discretization (\ref{eq_v_f}) for \(x^{(k)} := x_{\alpha^{(k)},h}\).
      \STATE Compute the data discrepancy
      \[
        d^{(k)} \;=\;|\,K\,x^{(k)} - {\boldsymbol w}|_n,
        \quad
        N^{(k)} \;=\;\|x^{(k)}\| \;+\; n^{-\tfrac12}\,d^{(k)}.
      \]
      \STATE Update the regularization parameter by
      \[
        \bigl(\alpha^{(k+1)}\bigr)^{\tfrac12 + \tfrac{1}{4m}}
        \;=\;
        C\,n^{-\tfrac12}\,\frac{d^{(k)}}{N^{(k)}},
      \]
      where \(C>0\) is a user–specified constant.
    \ENDFOR
  \end{algorithmic}
\end{algorithm}
A natural and easily computed starting value is $\alpha^{(0)} = n^{-\frac{2m}{2m+1}},$ which requires neither the best approximate solution norm $\|x^\dagger\|$ nor the noise standard variance $\sigma$.  This initialization has been adopted throughout our numerical experiments. Theorem \ref{algorithm} below proves that the sequence $\{\alpha^{(k)}\}$ produced by Algorithm \ref{alg:adpa-alpha} converges monotonically. 
Furthermore, our subsequent experimental results, detailed in subsection \ref{subsection5.2}, demonstrate that this convergence is achieved at a rapid rate, thereby highlighting the computational efficiency of the proposed algorithm in practical applications.

\begin{theorem}\label{algorithm}
    The iterates $\{\alpha^{(k)}\}_{k=0}^\infty$ produced by Algorithm \ref{alg:adpa-alpha} form a monotonically convergent sequence.

\end{theorem}
\begin{proof}
Without loss of generality, let $0<\alpha^{(k-1)}\le \alpha^{(k)}$.  Since \(x_{\alpha^{(k-1)},h}\) minimizes the functional (\ref{min_fun_h}), we have
\begin{equation}\label{ineq_1}
    |Kx_{\alpha^{(k-1)},h}-\boldsymbol  w|_n^2 +\alpha^{(k-1)}\,\|x_{\alpha^{(k-1)},h}\|^2
\;\le\;
|Kx_{\alpha^{(k)},h}-\boldsymbol w|_n^2 +\alpha^{(k-1)}\,\|x_{\alpha^{(k)},h}\|^2
\end{equation}
Likewise, for $x_{\alpha_{k,h}}$, we have
\begin{equation}\label{ineq_2}
    |Kx_{\alpha^{(k)},h}-\boldsymbol  w|_n^2 +\alpha^{(k)}\,\|x_{\alpha^{(k)},h}\|^2
\;\le\;
|Kx_{\alpha^{(k-1)},h}-\boldsymbol w|_n^2 +\alpha^{(k)}\,\|x_{\alpha^{(k-1)},h}\|^2
\end{equation}
Adding these two inequalities yields
\begin{equation}
    (\alpha^{(k-1)}-\alpha^{(k)})(\|x_{\alpha^{(k-1)},h}\|^2- \|x_{\alpha^{(k)},h}\|^2) \;\le\;0,
\end{equation}
which implies that \(\|x_{\alpha,h}\|\) is monotonically decreasing with respect to \(\alpha\).
Next, dividing the first inequality (\ref{ineq_1}) by \(\alpha^{(k-1)}\) and the second inequality (\ref{ineq_2}) by \(\alpha^{(k)}\), and then summing the results, we obtain
\begin{equation}
    \frac{\alpha^{(k-1)} -\alpha^{(k)}}{\alpha^{(k-1)}\alpha^{(k)}}(|Kx_{\alpha^{(k)},h}-\boldsymbol  w|_n^2 -|Kx_{\alpha^{(k-1)},h}-\boldsymbol  w|_n^2) \le 0,
\end{equation}
which implies that \(|Kx_{\alpha,h}-\boldsymbol  w|_n^2\) is monotonically increasing with respect to \(\alpha\).
Hence for \(\alpha^{(k-1)} \le \alpha^{(k)}\), we have
\begin{equation*}
    \frac{d^{(k-1)}}{N^{(k-1)}} \le \frac{d ^{(k)}}{N^{(k)}}.
\end{equation*}
Since the update reads
\begin{equation}\label{update}
    \bigl(\alpha^{(k+1)}\bigr)^{\frac12+\frac1{4m}} \;=\; C\,n^{-\tfrac12}\,\frac{d^{(k)}}{N^{(k)}},
\end{equation}
it follows from the exponent \(\frac{1}{2}+\frac{1}{4m} > 0\) that $\alpha^{(k)} \;\le\;\alpha^{(k+1)}.$
Therefore, the sequence $\{\alpha^{(k)}\}$ is monotonic. We observe from (\ref{update}) that \(0< \alpha^{(k)} < C\), establishing that the iterates $\{\alpha^{(k)}\}$ monotonically convergent.
\end{proof}

\subsection{Numerical results and discussions}\label{subsection5.2}
\begin{example}[\cite{neubauer2017}]
\label{example_H}
Consider the following linear integral equation of the first kind:
\begin{equation}\label{numerical_e_H}
    Kx(s) = \int_0^1 k(s, t)x(t) \, dt = y(s)
\end{equation}
with kernel
\[
k(s, t) = 
\begin{cases} 
s(1 - t), & s \leq t, \\
t(1 - s), & s \geq t.
\end{cases}
\]
Taking \( X = Y = L^2(0, 1) \), the integral operator \(K\) is compact, self-adjoint and injective with
\[
Kx = y \iff y \in H^2(0, 1) \cap H_0^1(0, 1) \text{ and } x = -y''.
\]
Thus, we take \(m=2\). The operator \(K\) has the eigensystem
\[
K\phi_j = s_j \phi_j, \quad s_j = (j\pi)^{-2}, \quad \phi_j(t) = \sqrt{2} \sin(j\pi t).
\]
In addition, Using interpolation theory (see \cite{Lions1972}), it can be demonstrated that for \( 4\mu - \frac{1}{4} \notin \mathbb{N}_0 \),
\begin{equation}\label{ex_spectral_K}
    \mathcal{R}((K^*K)^\mu) = \{f \in H^{4\mu}(0, 1) : f^{(2l)}(0) = 0 = f^{(2l)}(1), \, l = 0, 1, \ldots, \lfloor 2\mu - \frac{1}{4} \rfloor \}.
\end{equation}
In particular, when $\mu=\frac{1}{4}$, we have $\mathcal{R}\bigl((K^*K)^{\frac{1}{4}}\bigr) = H_0^1(0,1)$. Consequently, the associated spaces can be identified as \(W = H_0^1(0,1)\) and \(W^* = H^{-1}(0,1)\).
\end{example}

Based on the Example \ref{example_H}, we now present numerical experiments designed to validate the theoretical findings discussed in Section \ref{sec3} and \ref{sec4}. In these simulations, the noise sequence $\{e(s_i)\}_{i=1}^n$ is modeled as independent, identically distributed (i.i.d) Gaussian random variables with zero mean and standard deviation $\sigma$, which can also be viewed as a sub-Gaussian random variable with parameter $\sigma$. Let \(\sigma =  \delta  \|Kx^\dagger\|_{L^{\infty}(0,1)}\). We refer to \(\delta\) as the relative noise level. All numerical experiments below were tested using a mesh size \(h = 0.02\), which is sufficiently small to ensure that the FEM discretization error can be neglected.

\begin{itemize} 

    \item \textbf{Show that the a priori parameter choice (\ref{priori_parameter}) is nearly optimal:} the empirical errors \(|Kx_{n,\alpha}-Kx^\dagger|_n\) and \(\|x_{n,\alpha}-x^\dagger\|_{H^{-1}}\) almost reach its minimum when $\alpha$ is chosen closest to the prescribed a priori value. In the numerical experiment, the best-approximation solution of equation \eqref{numerical_e_H} is taken as \(x^\dagger(t) = -6t^{2}(1-t)(2 - 8t + 7t^{2})\), while the number of sampling points is fixed at \(n = 1{,}000\) and the relative noise level is set to \(\delta = 0.1\). Using the rule (\ref{priori_parameter}), the predicted optimal regularization parameter is \( \alpha^* \approx 1.06 \times 10^{-6} \), where the hidden constant in the $\mathcal{O}(\cdot)$ notation is taken to be 1. We evaluate the empirical error \(|Kx_{n, \alpha} - Kx^\dagger|_n\) and \(H^{-1}\)-norm error \(\|x_{n,\alpha}-x^\dagger\|_{H^{-1}}\) for the regularization parameter \(\alpha\) chosen from the set \(\{10^{-9}, 10^{-8}, 10^{-7}, 10^{-6}, 10^{-5}, 10^{-4}, 10^{-3}\}\). The obtained results are then compared with the empirical error \(|Kx_{n,\alpha^*} - Kx^\dagger|_n\) and the \(H^{-1}\)-norm error \(\|x_{n,\alpha^*} - x^\dagger\|_{H^{-1}}\) corresponding to the predicted parameter \(\alpha^*\). Figure \ref{fig:opt-alpha} demonstrates that the rule (\ref{priori_parameter}) provides a nearly optimal a priori parameter choice.




\begin{figure}[hbt!]
	\centering
	\setlength{\tabcolsep}{0pt}
	\begin{tabular}{cc}
		\includegraphics[width=0.3\textwidth]{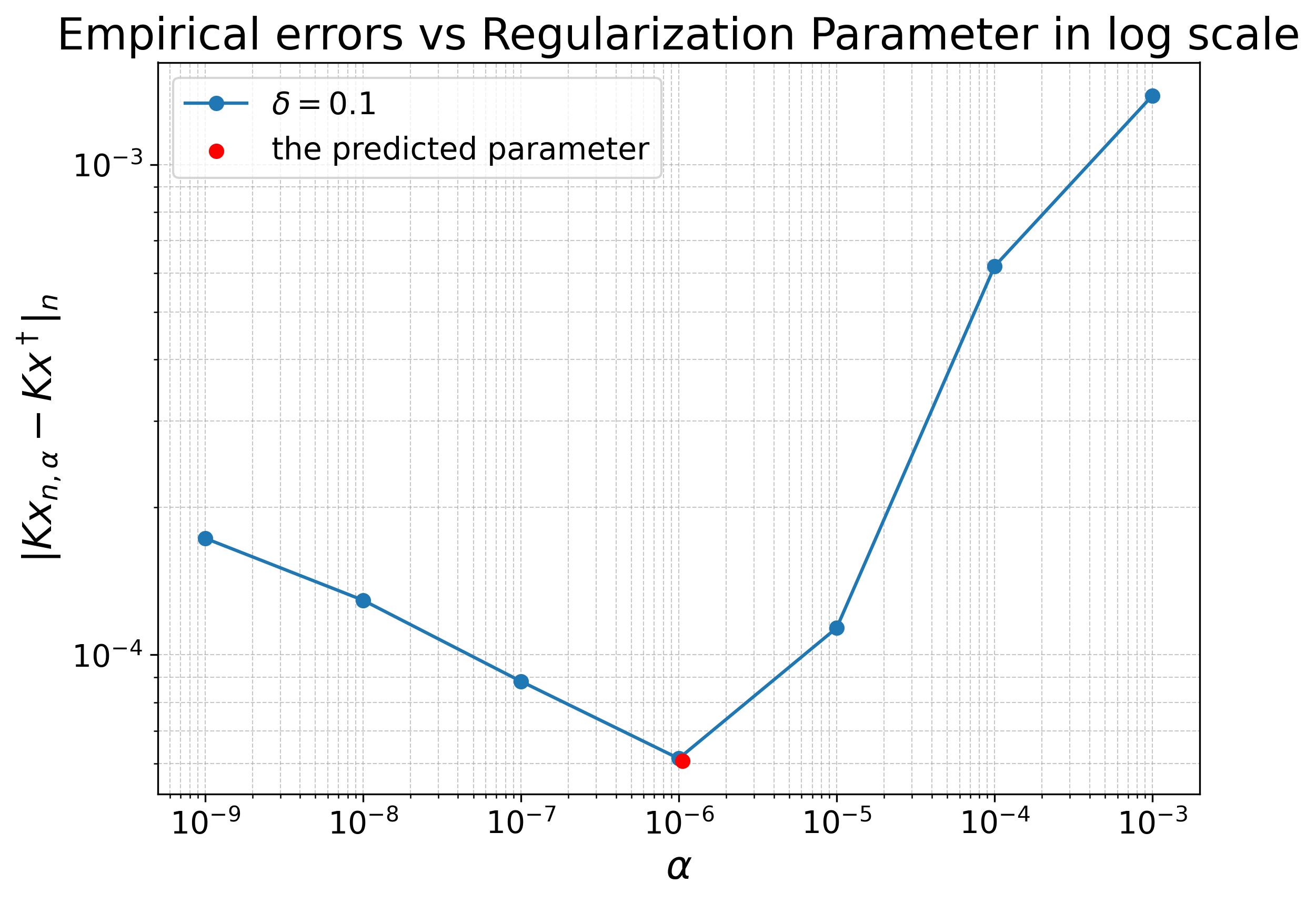} &\
		\includegraphics[width=0.3\textwidth]{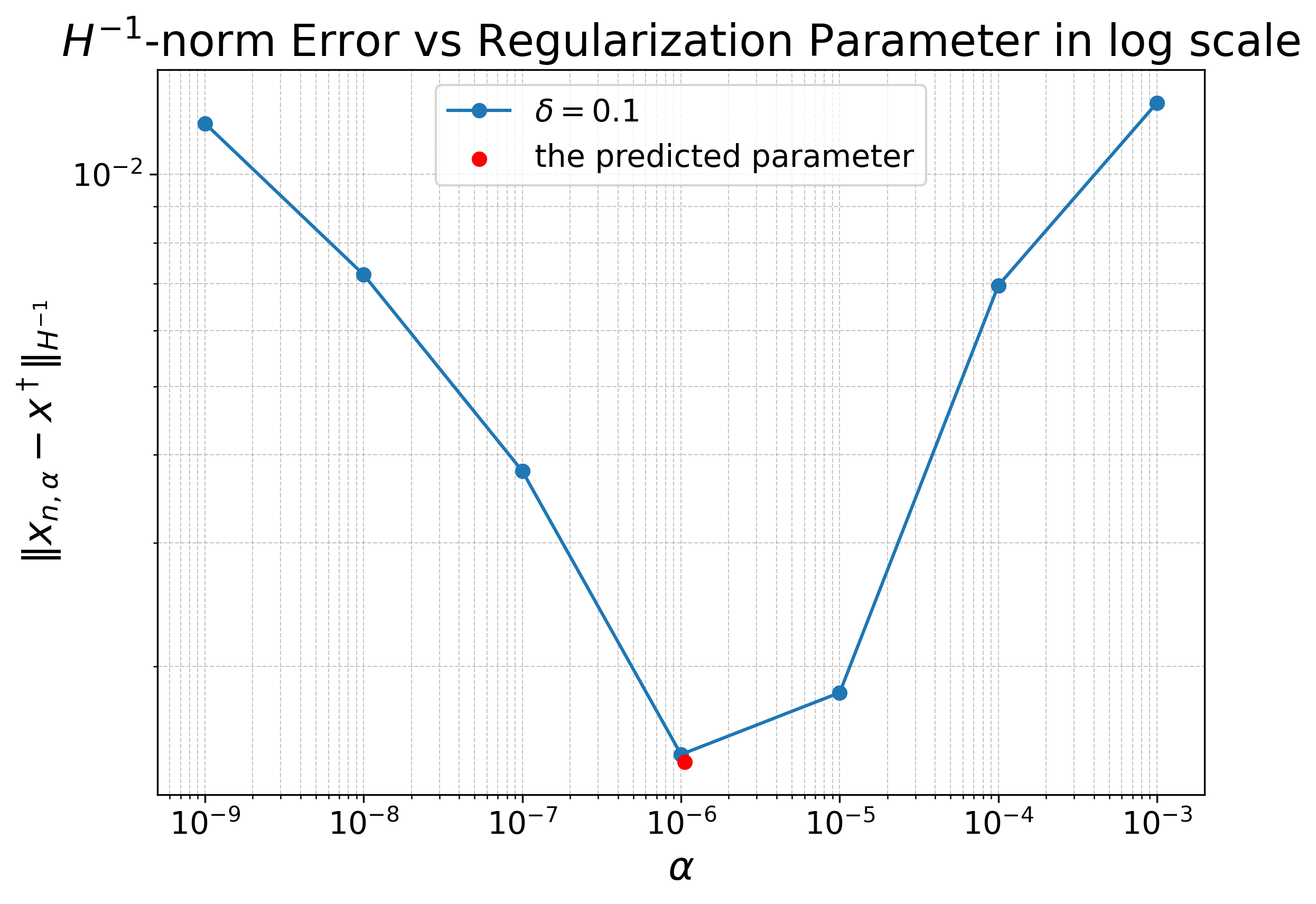} \\
	\end{tabular}
	\caption{Empirical error $|Kx_{n,\alpha}-Kx^\dagger|_n$ and $H^{-1}$-norm error $\|x_{n,\alpha}-x^\dagger\|_{H^{-1}}$ plotted against the regularization parameter $\alpha$ at the relative noise level \(\delta = 0.1\), highlighting the near‐optimal choice $\alpha^*$ obtained by the rule (\ref{priori_parameter}).}
	\label{fig:opt-alpha}
\end{figure}
\item \textbf{Verify the Corollary \ref{E_Kx} and Corollary \ref{E_x_n}}. In this case, we aim to verify the theoretical convergence rates
$$
\mathbb{E}\big[|Kx_{n,\alpha} - Kx^\dagger|^2_n\big] = O\!\left((\sigma n^{-1/2})^{\tfrac{8}{5}}\|x^\dagger\|^{\tfrac{2}{5}}\right), 
\quad
\mathbb{E}\!\left[\|x_{n,\alpha} - x^\dagger\|_{H^{-1}}^2\right] = O\!\left((\sigma n^{-1/2})^{\tfrac{4}{5}}\|x^\dagger\|^{\tfrac{6}{5}}\right).
$$

For each $(n,\sigma)$ pair, we conduct independent Monte Carlo trials. In each trial we add Gaussian noise with standard deviation $\sigma$ at the $n$ sampling points, and then evaluate the empirical semi-norm $|Kx_{n,\alpha} - Kx^\dagger|^2_n=\frac{1}{n}\sum_{i=1}^n\bigl((Kx_{n,\alpha}-Kx^\dagger)(s_i)\bigr)^2.$ The sample mean of these Monte Carlo realizations is used as an estimator of $\mathbb{E}[|Kx_{n,\alpha} - Kx^\dagger|^2_n]$.

The first suite of experiments used $n\in\{10{,}000,20{,}000,40{,}000,80{,}000\}$ and $\sigma\in\{0.005,0.01,0.05,0.1\}$ with $MC=10{,}000$ repetitions. After normalizing by $\|x^\dagger\|^{2/5}$, we fit the linear model in log–log coordinates,
$\log\big(\mathbb{E}[|Kx_{n,\alpha} - Kx^\dagger|^2_n]/\|x^\dagger\|^{2/5}\big) \approx a + b\log\eta,$
where $\eta=\sigma n^{-1/2}$.
As shown in the left panel of Figure \ref{fig:expecation}, the estimated slope is $b = 1.5804$, in close agreement with the theoretical value $\frac{8}{5}=1.6$. The second suite uses the same parameter ranges and number of repetitions. An analogous log–log regression of the normalized empirical expectations produced an estimated exponent \(b = 0.8001\), matching remarkably well with the theoretical prediction \(\frac{4}{5} = 0.8\) (see right panel of Figure \ref{fig:expecation}).

 \begin{figure}[hbt!]
	\centering
	\setlength{\tabcolsep}{0pt}
	\begin{tabular}{cc}
		\includegraphics[width=0.3\textwidth]{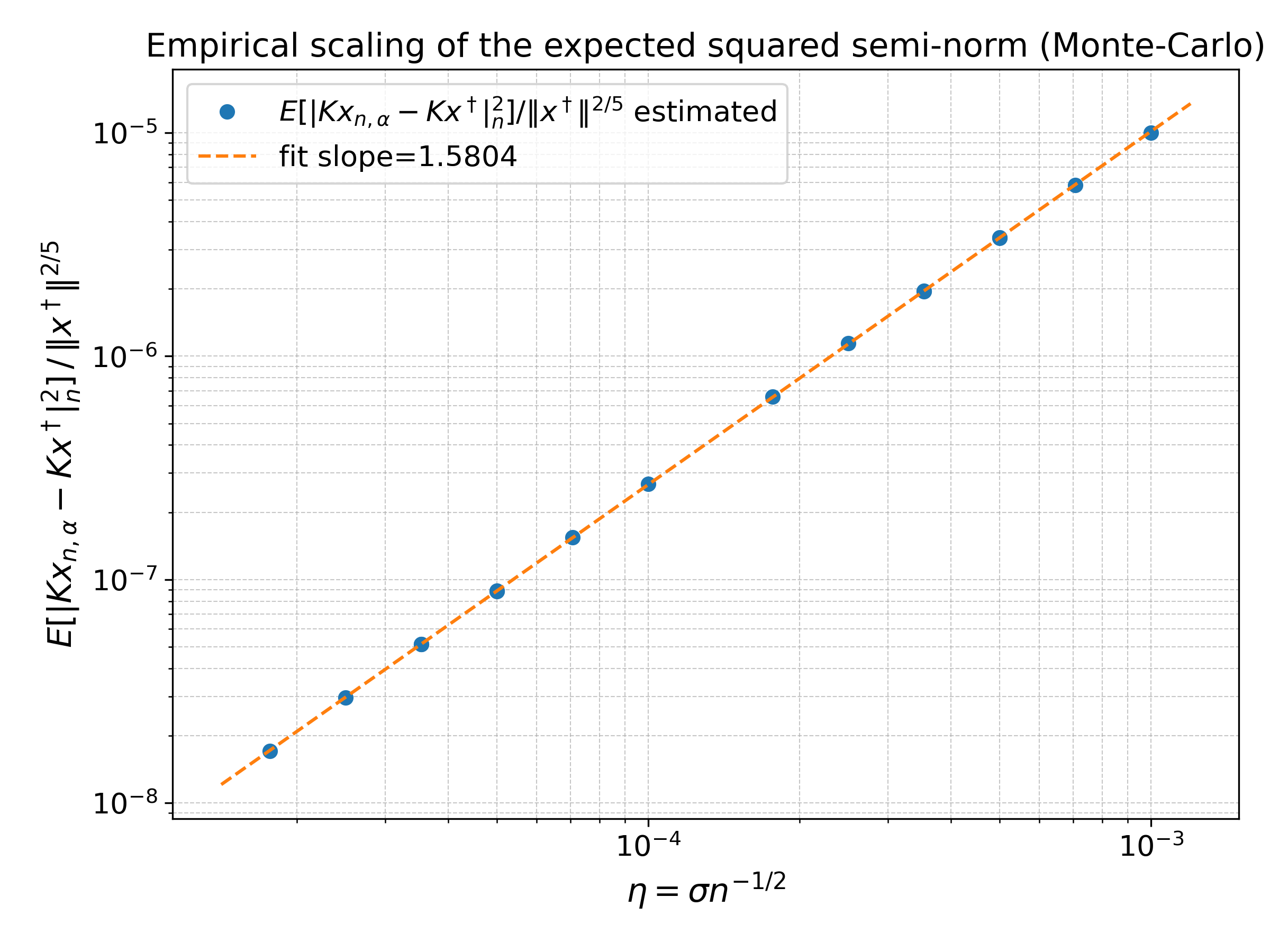} &\
		\includegraphics[width=0.3\textwidth]{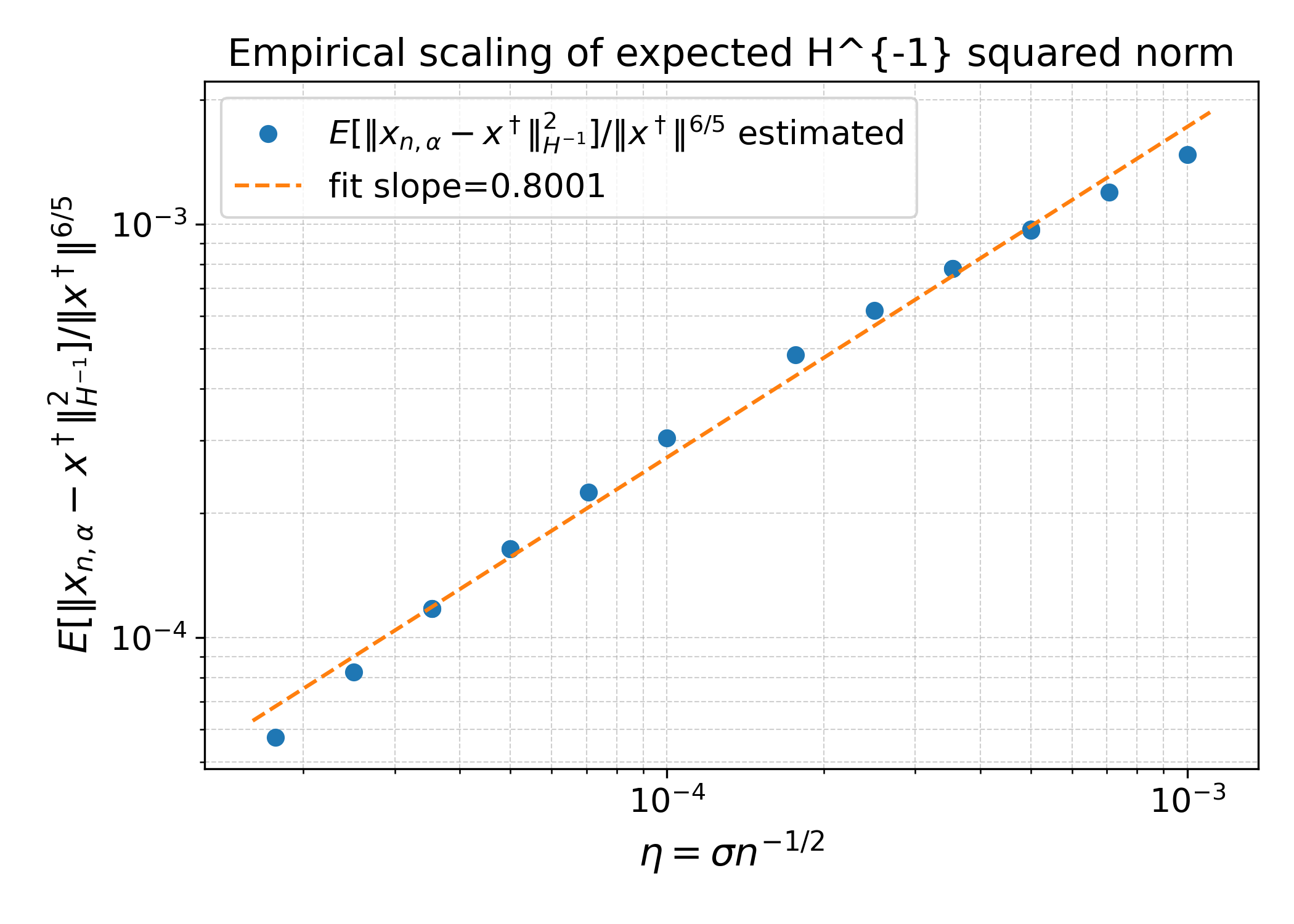} \\     
	\end{tabular}
	\caption{Log–log plots of the normalized empirical expectations versus the combined noise–sample parameter $\eta=\sigma n^{-1/2}$.
Left: $\mathbb{E}[|Kx_{n,\alpha}-Kx^\dagger|_n^2]/\|x^\dagger\|^{2/5}$ with least-squares fit slope $b\approx1.5804$ (theoretical 1.6).
Right: $\mathbb{E}[|x_{n,\alpha}-x^\dagger|_{H^{-1}}^2]/\|x^\dagger\|^{6/5}$ with least-squares fit slope $b\approx0.8001$ (theoretical 0.8).
Markers denote empirical means; dashed lines denote fitted power-law trends.}
	\label{fig:expecation}
\end{figure}

\item \textbf{Verify the exponential‐tail behavior predicted by Theorem \ref{Main_th3}.} To this purpose, we conduct a Monte Carlo experiment with \(\delta=0.01\), \(x^\dagger(t) = -6t^{2}(1-t)(2 - 8t + 7t^{2})\), and \(n=40{,}000\), yielding a nearly optimal parameter \(\alpha^* \approx 1.39\times10^{-9}\). We generate 10,000 independent noise realizations, compute the corresponding reconstructions \(x_{n,\alpha^*}\), and record the empirical errors \(|Kx_{n,\alpha^*}-Kx^\dagger|_n\). The left panel of Figure \ref{fig:histogram_qq} shows the histogram of these errors, revealing a rapidly decaying frequency with increasing magnitude. The right panel presents a Q–Q plot against the standard normal distribution, where the points closely follow the reference line, including in the tails. This indicates that the empirical error distribution is well approximated by a Gaussian law with exponentially decaying tails.

\begin{figure}[hbt!]
	\centering
	\setlength{\tabcolsep}{0pt}
	\begin{tabular}{cc}
		\includegraphics[width=0.6\textwidth]{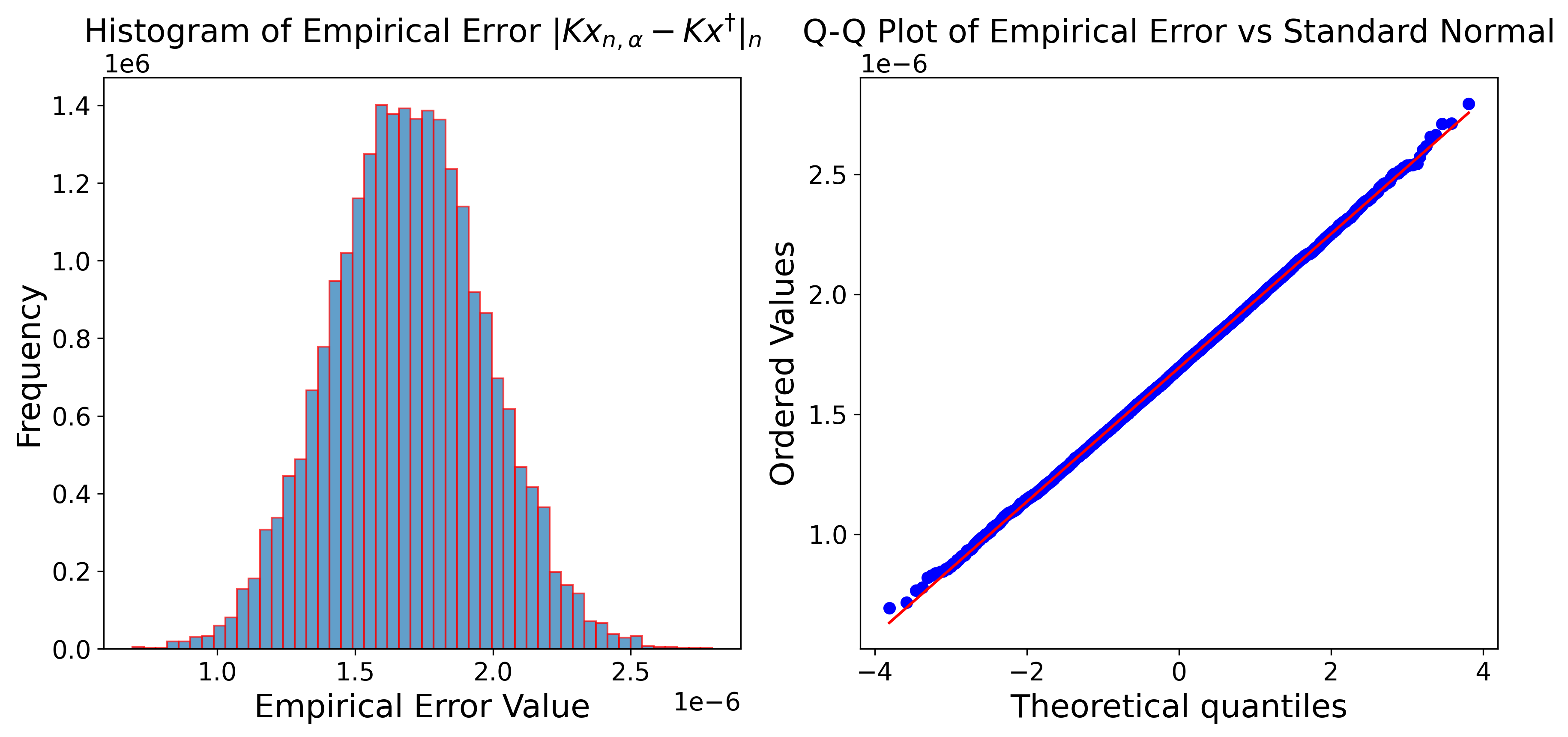} 
	\end{tabular}
	\caption{The histogram (left) illustrates the empirical errors \(|Kx_{n,\alpha^*} - Kx^\dagger|_n\) across 10,000 samples, while the quantile-quantile plot (right) shows the distribution of these 10,000 samples.}
	\label{fig:histogram_qq}
\end{figure}

\item \textbf{Show the efficiency of Algorithm \ref{alg:adpa-alpha} and verify Theorem \ref{algorithm}}. We consider the following two cases:

(a) \(x^\dagger(t) = -6t^2(1 - t)(2 - 8t + 7t^2)\), \(\delta = 0.001\) and \(n = 9,000\);

(b)  \(x^\dagger(t) = 
\begin{cases} 
0, & 0 \leq t \leq 1/2, \\
1, & 1/2 < t \leq  1.
\end{cases}\), \(\delta = 0.0001\) and \(n = 6,000\).

For case (a) and case (b), the noise standard deviations are \(\sigma = 8.39 \times 10^{-6}\) and \(\sigma = 7.03 \times 10^{-6}\), respectively. The initial guess for the regularization parameter is set to $\alpha^{(0)} = n^{-\frac{4}{5}}$. Algorithm \ref{alg:adpa-alpha} terminates either when the relative update of the regularization parameter satisfies \(|\frac{\alpha^{(k)}-\alpha^{(k+1)}}{\alpha^{(k+1)}}| < 10^{-3},\) or when the maximum number of iterations, set at 15, is reached. Figure \ref{fig:algorithm} illustrates the performance of Algorithm 1 from the first iteration onward.  The left panel (a) and (e) depict the strict monotonic decrease of the sequence $\{\alpha^{(k)}\}$, which converges in around five iterations under the convergence criterion $|\frac{\alpha^{(k)}-\alpha^{(k+1)}}{\alpha^{(k+1)}}| < 10^{-3}.$ Such rapid convergence underscores the algorithm’s computational efficiency. The adaptive procedure yields $\alpha^*\approx 1.11\times10^{-10}$ for case (a) and $\alpha^*\approx 9.60\times10^{-12}$ for case (b). The corresponding reconstructed solution is shown in Figure \ref{fig:algorithm} (d) and (h), and the associated \(L^2\)-norm relative errors for the discrete approximations are
\(\|x_{n,\alpha}-x^\dagger\|/\|x^\dagger\| \approx 5.12 \times 10^{-2}\) and \(5.67 \times 10^{-2},\) respectively. These numerically obtained regularization parameters are in close agreement with the theoretical value $1.15\times10^{-10}$ for case (a) and $9.41\times 10^{-12}$ for case (b), derived from the a priori rule \eqref{fig:algorithm} with the hidden constant set to 1. Moreover, the right panels (c) and (d) show the final data discrepancy $|K\,x_{\alpha^*,h}-w|_n \approx 8.24 \times 10^{-6}$ and $7.11\times 10^{-6}$, which serves as a reliable empirical approximation of the noise standard deviation $\sigma$ for case (a) and case (b), respectively.
\begin{figure}[hbt!]
	\centering
	\setlength{\tabcolsep}{0pt}
	\begin{tabular}{cc}
		\includegraphics[width=0.68\textwidth]{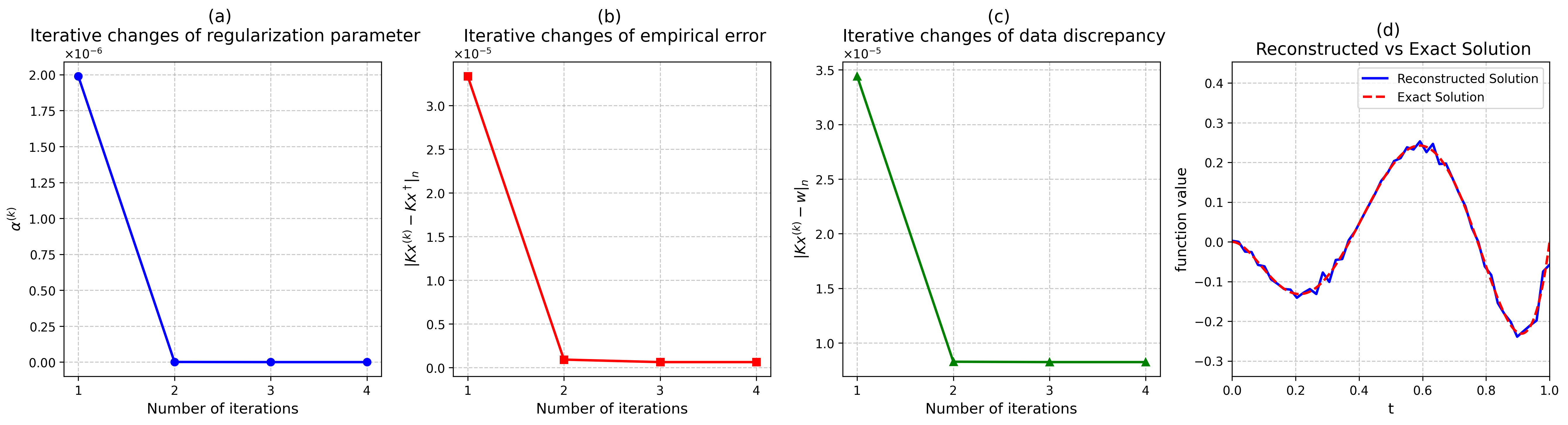} \\
		\includegraphics[width=0.68\textwidth]{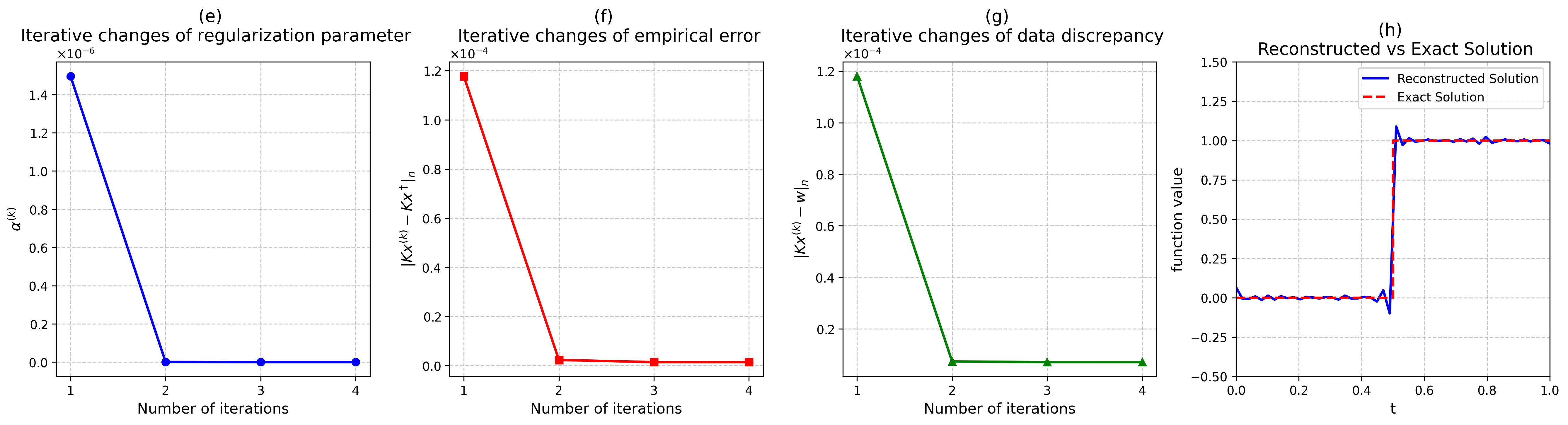} \\
	\end{tabular}
	\caption{(a) and (e) show the evolution of the regularization parameter \(\alpha\); (b) and (f) present the empirical error \( |Kx_{\alpha,h} - Kx^\dagger|_n \); and (c) and (g) show the data discrepancy \( |Kx_{\alpha,h} - w|_n \), all plotted versus the iteration count for case (a) and case (b), respectively. Panels (d) and (h) provide a comparison between the reconstructed solutions obtained by Algorithm~\ref{algorithm} and the exact solution for the two cases. }
	\label{fig:algorithm}
\end{figure}

\end{itemize}

\begin{rem}
    If the underlying solution is known a priori to possess sufficient smoothness, it is natural to employ an \(H^1\)-based regularization of the form
\[
\min_{x\in L^2(a,b)} J_{n,\alpha}(x)
= |Kx - \boldsymbol w|_n^2 + \alpha \|x\|_{H^1}^2.
\]
Under such circumstances, the enhanced smoothness penalty provides a more accurate and stable reconstruction than a purely \(L^2\)-based approach. Moreover, the analytical framework developed in this work extends directly to the \(H^1\)-regularized setting, since the core variational structure and stability mechanisms remain valid after replacing the zeroth-order penalty with its first-order counterpart.
\end{rem}

\begin{example}
    Let
\[
K:L^{2}(0,1)\to L^{2}(0,1),\qquad (Kx)(s)=\int_{0}^{1} e^{-|s-t|}x(t)\,\mathrm{d}t.
\]
\end{example}
It can be verified that the range of the operator $K$ is given by
$$
R(K) \;=\; \bigl\{\, y \in H^2(0,1) \;:\; y'(0)=y(0), \;\; y'(1)=-y(1) \,\bigr\},
$$
Thus, we take \(m = 2\). We now provide numerical evidence that the singular values of the integral operator \(K\) decay at the rate predicted by Theorem~\ref{s-val decay}. 
The interval \([0,1]\) is discretized using a uniform mid-point quadrature rule with \(N\) subintervals, where the nodes are given by \(t_j=(j-0.5)/N\). 
This yields the symmetric matrix
\[
A_{ij} = h\,\exp(-|s_i-t_j|), \qquad h=1/N,
\]
which serves as a finite-dimensional approximation of \(K\) in the \(L^2\)-inner product. 
For each discretization size \(N\), we compute the eigenvalues of \(A\), ordered in descending magnitude, and interpret them as approximations of the singular values of \(K\). 

To quantify the decay rate, we fit the discrete singular values to a power law of the form \(\log s_j \approx a + b \log j,\) using linear regression on a log–log scale. The fitting is performed over a mid-range index set \(j \in [j_0,j_1]\), in which the first few dominant singular values are excluded since they are not representative of the asymptotic decay, and the extreme tail is omitted to avoid discretization and round-off effects. In the experiments, we chose \(j_0=6\) and \(j_1=\min(400,\lfloor N/2 \rfloor)\). Figure~\ref{fig:s_v} presents the results for several discretization sizes. In all cases, the estimated slope is approximately \(-2\), which is consistent with Theorem \ref{s-val decay}. 


\begin{figure}[hbt!]
	\centering
	\setlength{\tabcolsep}{0pt}
	\begin{tabular}{cc}
		\includegraphics[width=0.6\textwidth]{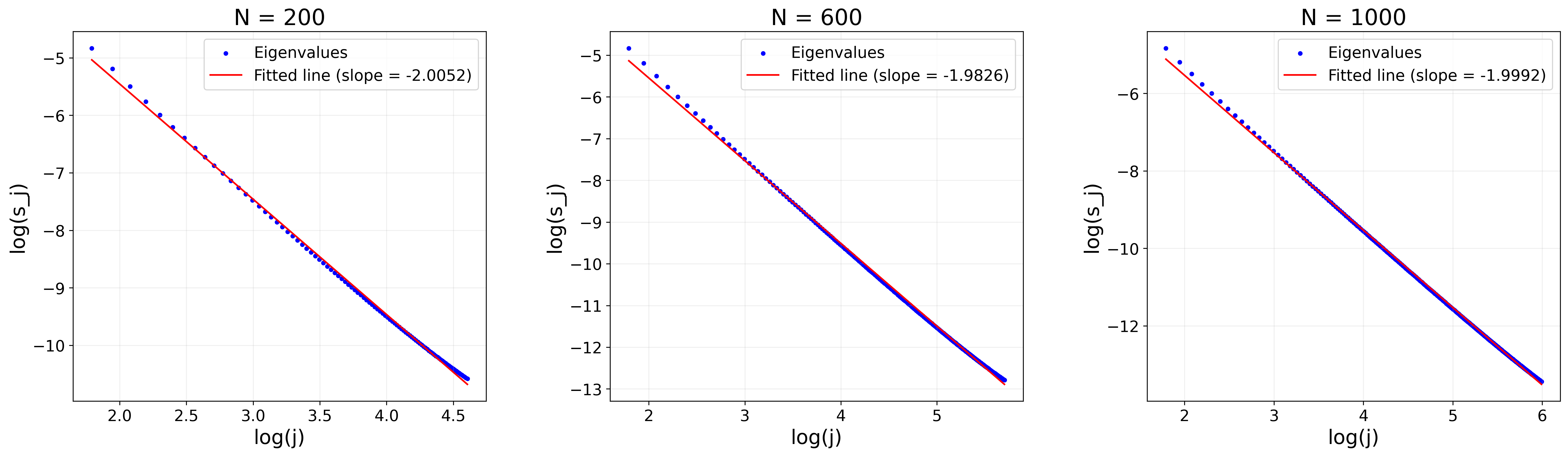} \\   
	\end{tabular}
	\caption{Fitted slopes of \(\log s_j\) vs.\ \(\log j\) for different discretization sizes.
}\label{fig:s_v}
\end{figure}


\bibliographystyle{plain}



\end{CJK}

\end{document}